\documentclass[letterpaper,12pt]{extarticle}
\usepackage[margin=1.2in,letterpaper]{geometry} 
\usepackage[dvips]{graphicx}
\usepackage{fancyhdr}
\usepackage{color}
\usepackage{theoremref, hyperref}
\usepackage{amsmath, amsthm}
\usepackage{amssymb}
\usepackage{tikz}
\usepackage{bm}
\usepackage{url}
\usepackage{latexsym}
\usepackage{arydshln}

\numberwithin{equation}{section}

\definecolor{darkcyan}{rgb}{0.0, 0.55, 0.55}

\newcommand{\Hil}[0]{
\mathcal{H}
}

\newcommand{\Rd}[0]{
{\mathbb{R}^d}
}

\newcommand{\B}[0]{{\mathcal{B}}}
\newcommand{\D}[0]{{\mathcal{D}}}

\newcommand{\A}[0]{{\mathcal{A}}}
\newcommand{\J}[0]{{\mathcal{J}}}
\newcommand{\MS}[0]{{\mathcal{S}}}
\newcommand{\C}[0]{{\mathcal{C}}}

\newtheorem{theorem}{Theorem}[section]
\newtheorem{definition}[theorem]{Definition}
\newtheorem{proposition}[theorem]{Proposition}
\newtheorem{lemma}[theorem]{Lemma}
\newtheorem{corollary}[theorem]{Corollary}
\newtheorem{remark}[theorem]{Remark}

\newtheorem{conj.}[theorem]{Conjecture}
\newtheorem{Bsp.}[theorem]{Example}

\begin{document}
%
%
\title{\bf\vspace{-10
pt} 
Wiener pairs of Banach algebras of operator-valued matrices}
%
%
\author{L. Köhldorfer and P. Balazs}

%
%
\date{}

%
%
\maketitle

\begin{abstract}
In this article we consider several new examples of Wiener pairs $\A \subseteq \B$, where $\B = \B(\ell^2(X;\Hil))$ is the Banach algebra of bounded operators acting on the Hilbert space-valued Bochner sequence space $\ell^2(X;\Hil)$ and $\A = \A(X)$ is a Banach algebra consisting of operator-valued matrices indexed by some relatively separated set $X \subset \Rd$. In particular, we consider $\B(\Hil)$-valued versions of  the Jaffard algebra, of certain weighted Schur-type algebras, of Banach algebras which are defined by more general off-diagonal decay conditions than polynomial decay, of weighted versions of the Baskakov-Gohberg-Sjöstrand algebra, and of anisotropic variations of all of these matrix algebras, and show that they are inverse-closed in $\B(\ell^2(X;\Hil))$. In addition, we obtain that each of these Banach algebras is symmetric.
\end{abstract}


\section{Introduction}

From a historical point of view, both the title and the content of this article originate from \emph{Wiener's Lemma on absolutely convergent Fourier series} \cite{Wiener32}, which states that if a continuous function $f\in C(\mathbb{T})$ on the torus $\mathbb{T}$ admits an absolutely convergent Fourier series 
and is nowhere vanishing, then the function $1/f$ admits an absolutely convergent Fourier series as well. In more abstract terms, Wiener's Lemma can be rephrased to a certain relation between the Banach algebra $\A(\mathbb{T})$ of absolutely convergent Fourier series and the Banach algebra $C(\mathbb{T})$ of continuous functions on $\mathbb{T}$. Indeed, $f\in \A(\mathbb{T}) \subseteq C(\mathbb{T})$ being nowhere vanishing is equivalent to $f\in \A(\mathbb{T})$ being invertible in the larger Banach algebra $C(\mathbb{T})$, and Wiener's Lemma guarantees that these two assumptions already imply that the continuous inverse $1/f \in C(\mathbb{T})$ of $f$ is contained in $\A(\mathbb{T})$ as well, i.e. that $f$ is also invertible in the smaller Banach algebra $\A(\mathbb{T})$. 

Nowadays, the above theme manifests in various kind of situations and degrees of abstraction. The general definition is the following. If $\A\subseteq \B$ is a nested pair of two (possibly non-commutative) Banach algebras $\A$ and $\B$ with common identity, then $\A$ is called \emph{inverse-closed} in $\B$ \cite{Barnes2000}, if 
\begin{equation}\label{Wienerpair}
A \in \A \, \text{ and } \, \exists A^{-1} \, \text{ in } \,  \B \qquad \Longrightarrow \qquad A^{-1} \in \A .    
\end{equation}
In this case we say that $\A\subseteq \B$ is a \emph{Wiener pair} \cite{naimark70,naimark72}. As in the example of the Wiener pair $\A(\mathbb{T}) \subseteq C(\mathbb{T})$, it is often easier to verify the invertibility of an element $A\in \A\subseteq \B$ in the larger algebra $\B$, than verifying its invertibility in the smaller algebra $\A$. Consequently, $\A$ being inverse-closed in $\B$ is a fairly strong property and thus immensely useful in a huge number of situations. In particular, Wiener's Lemma and its many generalizations, see e.g.  \cite{baskri14,ja90,barnes87,BochnerPhillips42,Baskakov1990,groelei06,Gohberg89,Sjöstrand1994-1995,GroechSjoe1,bakrus24,BalanKrishtal10,BaChKrOkRo14,ShiSun08,zbMATH06949773,sun07a,sun11,DaSo22,dasoka23,kur01}, appear in a vast number of applications such as in numerical analysis \cite{olestroh1,grrzst10,Strohmer1999,strohmer02}, approximation theory \cite{groeklotz10,corgroech04}, time-frequency and wavelet analysis \cite{ja90,corgroech04,ShiSun08}, sampling theory \cite{atbekaco03,corgroech04,BaChKrOkRo14,ALDROUBI20081667}, pseudo differential operators and differential equations \cite{Sjöstrand1994-1995,GroechSjoe1,groetoft10,zbMATH06949773}, or frames \cite{groe2003,gr04-1,forngroech1,bacahela06a,Balan2006,xxlgro14,ALDROUBI20081667}. For  more mathematical background and historical remarks on Wiener's Lemma and its variations we refer the reader to the survey \cite{gr10-2}, which is one of the main inspirations of this work. 


Our motivation for writing this article comes from studying localization properties of g-frames, which are a generalization of frames \cite{ole1n} to families of operators satisfying a frame-like inequality. More precisely, a countable family $T=(T_k)_{k\in X}$ of bounded operators $T_k\in \B(\Hil)$ on some separable Hilbert space $\Hil$ is called a \emph{g-frame} \cite{Sun2006437}, if there exist positive constants $0<A\leq B < \infty$, such that
\begin{equation}\label{gframeineq}
A \Vert f \Vert^2 \leq \sum_{k\in X} \Vert T_k f \Vert ^2 \leq B \Vert f \Vert^2 \qquad (\forall f\in \Hil).
\end{equation}
Every g-frame $(T_k)_{k\in X}$ allows for a linear, bounded and stable reconstruction of any vector $f\in \Hil$ from the vector-valued family $(T_k f)_{k\in X} \in \ell^2(X;\Hil)$, where
$$\ell^2(X;\Hil) = \left\lbrace (g_k)_{k\in X}: g_k \in \Hil \,\,  (\forall k\in X), (\Vert g_k \Vert)_{k\in X} \in \ell^2(X) \right\rbrace .$$
Beside their capability to achieve perfect reconstruction of vectors in $\Hil$, g-frames might satisfy other important properties which do not follow solely from the existence of the g-frame bounds $A$ and $B$. For instance, in \cite{skret2020}, g-frames for $\Hil = L^2(\mathbb{R}^d)$ of the form $(\pi(k) T \pi(k)^*)_{k\in X}$, called \emph{Gabor g-frames}, were considered, where $X = \Lambda \subseteq \mathbb{R}^{2d}$ is a full-rank lattice, $T \in \B(L^2(\mathbb{R}^d))$ is some suitable window operator, and $\pi(k) = M_{k_2} T_{k_1}$ denotes a time-frequency shift by $k=(k_1, k_2) \in \mathbb{R}^{2d}$ in the time-frequency plane \cite{gr01}. Via Fourier methods for periodic operators, which ultimately rely on the group structure of the lattice $\Lambda$, Skrettingland proved in \cite{skret2020} several results for this class of g-frames, which typically are concluded from suitable localization properties of a given frame in some abstract Hilbert space $\Hil$ \cite{forngroech1}. We are convinced that $-$ in analogy to localized frames \cite{forngroech1} $-$ these localization-type results do not rely on Fourier methods or even on the group structure of the index set $X$, but are a consequence of the intrinsic localization properties of the underlying g-frame $(T_k)_{k\in X}$, which are measured by the decay of the operator norms $\Vert T_k T_l^* \Vert$ of the entries $T_k T_l^* \in \B(\Hil)$ of the g-Gram matrix $G_T = [T_k T_l^*]_{k,l\in X} \in \B(\ell^2(X;\Hil))$. In fact, an article on g-frames, whose localization quality is measured by its g-Gram matrix belonging to some suitable Banach algebra $\A$ which is inverse-closed in the Banach algebra $\B = \B(\ell^2(X;\Hil))$ of bounded operators acting on the Hilbert space $\ell^2(X;\Hil)$ (see also \cite[Definition 5.7]{Krishtal2011}), is currently in preparation \cite{köba24inprep}. 

In view of the above motivation, the aim of this article is to prove several new examples of Wiener pairs $\A \subseteq \B(\ell^2(X;\Hil))$, where $\A = \A(X)$ is a Banach algebra of $\B(\Hil)$-valued matrices indexed by some relatively separated set $X\subset \mathbb{R}^d$. Beside the above mentioned possibility to generalize the concept of intrinsically localized frames from \cite{forngroech1} to the g-frame setting \cite{köba24inprep}, we believe that our new examples of such Wiener pairs provide a useful arsenal of techniques for operator theory \cite{douglas}, the study of Fourier series of operators \cite{feiko98}, or quantum harmonic analysis \cite{werner84,Gosson21}. 

The paper is structured as follows. In Section \ref{Notation and preliminary results} we make our notation precise and collect some preliminary facts which we will need later on. Each of the later sections is devoted to defining one specific class $\A$ of $\B(\Hil)$-valued matrices and showing that $\A$ is indeed a Banach algebra which is inverse-closed in $\B(\ell^2(X;\Hil))$. In particular, we prove these results for certain weighted Schur-type algebras $\MS_{\nu}^1(X)$ (see Section \ref{The Schur algebra}), for the Jaffard algebra $\J_{s}(X)$, which describes polynomial decay of order $s$ away from the diagonal and the Banach algebras $\J_{\nu}(X)$ or $\B_{u,s}(X)$ which are defined by more general off-diagonal decay conditions than polynomial decay (see Section \ref{Off-diagonal decay conditions via admissible weights}), for weighted versions of the Baskakov-Gohberg-Sjöstrand algebra $\C_v(\mathbb{Z}^d)$ of convolution dominated matrices (see Section \ref{The Baskakov-Gohberg-Sjöstrand algebra}), and for anisotropic variations of all of these Banach algebras (see Section \ref{Anisotropic decay conditions}). In particular, we prove that all of these Banach algebras are symmetric.

It should be pointed out that the Banach algebras $\MS_{\nu}^1(\mathbb{Z}^d)$, $\J_s(\mathbb{Z}^d)$ and $\C_{\nu}(\mathbb{Z}^d)$, and their inverse-closedness in $\B(\ell^2(\mathbb{Z}^d;\Hil))$ were already mentioned in the excellent overview article \cite{kri11}. However, we would like to emphasize that the corresponding proofs of those results in \cite{kri11} are on many occasions rather outlined than given in detail, and at the same time heavily depend on Fourier methods and other Banach algebraic methods which rely on the group structure of the index set $\mathbb{Z}^d$. Except for the Baskakov-Gohberg-Sjöstrand algebra $\C_v(\mathbb{Z}^d)$ and its anisotropic variation, all Banach algebras that are considered in the current article, however, are indexed by some arbitrary relatively separated index set $X\subset \mathbb{R}^d$ which does not necessarily possess a group structure. Investigating these more general matrix algebras seems only natural to us, since this will allow a systematic treatment of more general classes of localized g-frames, including irregular Gabor g-frames. It should be noted, however, that while the index set $X$ is not necessarily a group, the group structure of $\mathbb{R}^d$ is still inherently present in the problem due to the way matrix multiplication is defined. As noted in \cite[Remark 5.6]{zbMATH06949773}, one could also consider continuous analogs of the algebras $\MS_{\nu}^1(\mathbb{Z}^d)$, $\J_s(\mathbb{Z}^d)$ and $\C_{\nu}(\mathbb{Z}^d)$ (see, e.g., \cite{zbMATH06949773}), which are defined via a bounded uniform partition of unity (BUPU) indexed by $\mathbb{Z}^d$; and since the definition of each of the aforementioned algebras involves a \emph{balanced weight} (see Subsection \ref{Weight functions}), one may argue similarly as in \cite[Lemma 5.5]{zbMATH06949773} and deduce some of our results from the continuous setting. 

In this article, we pursue a more self-contained presentation by working with the definitions directly and employing different proof methods than the Fourier-based approaches mentioned in \cite{kri11}. 
In particular, our analysis is based on various kinds of Banach algebraic techniques, such as \emph{Hulanicki's Lemma} (see Proposition \ref{Hulanickilemma}), \emph{Brandenburg's trick} \cite{BRANDENBURG75}, a vector-valued version of \emph{Barnes' Lemma} \cite{barnes87} and adapted techniques from \cite{groelei06}, and approaches using derivation algebras \cite{groeklotz10}. 


\section{Notation and preliminary results}\label{Notation and preliminary results}

Throughout these notes, all the index sets considered are countable. We denote the cardinality of a set $X$ by $\vert X \vert$. For $x\in \mathbb{R}^d$, the symbol $\vert x\vert$ denotes the Euclidean norm on $\mathbb{R}^d$ and $\vert x \vert_{\infty}$ the maximum norm on $\mathbb{R}^d$. The standard scalar product of $x$ and $y$ in $\mathbb{R}^d$ is denoted by $x\cdot y = \sum_{j=1}^d x_j y_j$. 
The symbol $\mathbb{N}$ denotes the set of positive integers $\lbrace 1, 2, 3, \dots \rbrace$ and $\mathbb{N}_0 = \lbrace 0, 1, 2, 3, \dots \rbrace$ denotes the set of non-negative integers. As usual, for multi-indices $\alpha =(\alpha_j)_{j=1}^d, \beta =(\beta_j)_{j=1}^d \in \mathbb{N}_0^d$ we write $\vert \alpha \vert = \sum_{j=1}^d \alpha_j$, $\binom{\alpha}{\beta} = \prod_{j=1}^d \binom{\alpha_j}{\beta_j}$, $x^{\alpha} = x_1^{\alpha_1}\dots x_d^{\alpha_d}$ for $x\in \mathbb{R}^d$, and $\beta \leq \alpha$ if and only if $\beta_j \leq \alpha_j$ for all $1\leq j \leq d$. We write $\delta_{ij}$ for the Kronecker-delta. 
 
$\mathcal{H}$ is always a separable Hilbert space. 
The domain, kernel and range of an operator $T$ is denoted by $\D(T)$, $\mathcal{N}(T)$ and $\mathcal{R}(T)$ respectively. The identity element on a given space $X$ is denoted by $\mathcal{I}_X$. The set of bounded operators between two (quasi-)normed spaces $X$ and $Y$ is denoted by $\mathcal{B}(X,Y)$ and we set $\mathcal{B}(X) := \mathcal{B}(X,X)$. 

\subsection{Banach algebras}\label{Banach algebras}

We first recall some basic concepts from Banach algebra theory. 

A \emph{Banach algebra} is a (complex) Banach space $(\A, \Vert \, . \, \Vert_{\A})$ equipped with a bi-linear map $\cdot : \A \times \A \longrightarrow \A$, called \emph{multiplication}, which is associative and satisfies $\alpha (A\cdot B) = A\cdot (\alpha B) = (\alpha A)\cdot B$ as well as $\Vert A\cdot B \Vert_{\A} \leq \Vert A\Vert_{\A} \Vert B \Vert_{\A}$ for all $A,B \in \A, \alpha \in \mathbb{C}$. We call $\A$ \emph{unital} if $\A$ contains a neutral element $\mathcal{I}_{\A}$ with respect to multiplication. An element $A$ of a unital Banach algebra $\A$ is called \emph{invertible}, if there exists an element $A^{-1} \in \A$ such that $A\cdot A^{-1} = A^{-1} \cdot A = \mathcal{I}_{\A}$. An involution on $\A$ is a continuous conjugate linear map $^*:\A \longrightarrow \A$ satisfying $(A^*)^* = A$ and $(A\cdot B)^* = B^* \cdot A^*$ for all $A,B\in \A$. If $\A$ possesses an involution which is an isometry, then $\A$ is called a \emph{Banach *-algebra}. 

Let $\A$ be a unital Banach algebra. For any $A$ in $\A$, the \emph{resolvent set} of $A$ is defined by $\rho_{\A}(A) := \lbrace \lambda \in \mathbb{C}: \lambda \mathcal{I}_{\A} - A \, \, \text{is invertible}\rbrace$. The \emph{spectrum} of $A$ is the compact set defined by $\sigma_{\A}(A):= \mathbb{C}\setminus \rho_{\A}(A)$. The \emph{spectral radius} of $A$ is defined by $r_{\A}(A) = \max \lbrace \vert\lambda\vert: \lambda \in \sigma_{\A}(A) \rbrace$. By \emph{Gelfand's formula}, the spectral radius of any $A\in \A$ can be computed via
\begin{equation}\label{Gelfand}
    r_{\A}(A) = \lim_{n\rightarrow \infty} \Vert A^n \Vert_{\A}^{\frac{1}{n}}. 
\end{equation}
If $\A$ and $\B$ are two unital Banach algebras (Banach *-algebras) with a common identity, then we write $\A\subseteq \B$, if $\A$ is contained in $\B$ and if the inclusion map from $\A$ into $\B$ is an algebra homomorphism (*-homomorphism). If $\A$ and $\B$ are two unital Banach algebras with common identity, then \cite{gr10-2}
\begin{equation}\label{spectrum1}
    \A \subseteq \B \qquad \Longrightarrow \qquad \sigma_{\B}(A) \subseteq \sigma_{\A}(A) \quad (\forall A\in \A).
\end{equation}
Such a nested pair $\A \subseteq \B$ of unital Banach algebras with common identity is called a \emph{Wiener pair} if $\A$ is \emph{inverse-closed} in $\B$, see (\ref{Wienerpair}). One can show \cite{gr10-2} that a nested pair $\A \subseteq \B$ of unital Banach algebras with common identity forms a Wiener pair if and only if $\sigma_{\A}(A) = \sigma_{\B}(A)$ for all $A\in \A$. This is the reason why the term \emph{spectral invariance} is often used synonymously when discussing inverse closedness \cite{gr10-2}. 

Verifying whether $\A \subseteq \B$ forms a Wiener pair is often a rather subtle challenge. However, if $\A$ and $\B$ are Banach *-algebras and $\B$ is \emph{symmetric}, which means that $\sigma_{\B}(A^* A) \subseteq [0,\infty )$ for all $A\in \B$, a powerful trick by Hulanicki allows an analytical treatment of this task. 

\begin{proposition}[Hulanicki's Lemma]\label{Hulanickilemma}\cite{Hulanicki1972}
Let $\mathcal{A}\subseteq \mathcal{B}$ be a pair of unital Banach *-algebras with common identity, and suppose that $\mathcal{B}$ is symmetric. Then the following are equivalent:
\begin{itemize}
\item[(i)] $\mathcal{A}$ is inverse-closed in $\mathcal{B}$.
\item[(ii)] $r_{\mathcal{A}}(A)=r_{\mathcal{B}}(A) \qquad (\forall A=A^* \in \mathcal{A})$.
\item[(iii)] $r_{\mathcal{A}}(A)\leq r_{\mathcal{B}}(A) \qquad (\forall A=A^* \in \mathcal{A})$.
\end{itemize}
Moreover, if one of these conditions holds, then $\A$ is symmetric as well. 
\end{proposition}

We will also need some facts from derivation algebras. Let $\A$ be a symmetric Banach algebra and $\delta: \D(\delta) = \D(\delta, \A) \longrightarrow \A$ be a closed (possibly unbounded) linear operator, where the domain $\D(\delta) = \D(\delta, \A)$ is an arbitrary subspace of $\A$. Such an operator $\delta$ is called a \emph{derivation on $\A$}, if the Leibniz rule holds true on $\D(\delta)$, i.e. if
\begin{equation}\label{Leibniz}
    \delta(A\cdot B) = A \cdot \delta(B) + \delta(A) \cdot B \qquad (\forall A,B \in \D(\delta)).
\end{equation}
If $\A$ is involutive, we additionally assume that the domain $\D(\delta)$ is invariant under involution (i.e. $A^* \in \D(\delta)$ whenever $A\in \D(\delta)$) and that $\delta(A^*) = \delta(A)^*$ for all $A\in \D(\delta)$. Equipped with the graph norm 
\begin{equation}
\Vert A \Vert_{\D(\delta)} = \Vert A \Vert_{\A} + \Vert \delta(A) \Vert_{\A},   
\end{equation}
$\D(\delta)$ then forms a (not necessarily unital) symmetric Banach (*-)algebra whenever $\A$ is a symmetric unital Banach (*-)algebra \cite[Theorem 3.4]{groeklotz10}.

\subsection{Weight functions}\label{Weight functions}

A \emph{weight function} on $\mathbb{R}^{d}$, or simply a \emph{weight}, is a continuous and positive function $\nu :\mathbb{R}^{d} \longrightarrow (0,\infty )$. In the sections ahead we will consider weighted versions of various kinds of matrix algebras. The following classes of weight functions are fundamental for the study of these matrix algebras. 

A weight $\nu$ is called \emph{submultiplicative}, if 
$$\nu (x+x') \leq \nu(x)\nu(x') \qquad (\forall x, x' \in \mathbb{R}^{d}),$$
and \emph{symmetric} if 
$$\nu(x) = \nu(-x) \qquad (\forall x \in \mathbb{R}^d).$$
A weight $m$ is called \emph{$\nu$-moderate} if there exists some constant $C>0$ such that 
$$m(x+x') \leq Cm(x)\nu(x') \qquad (\forall x, x' \in \mathbb{R}^{d}).$$
A weight $\nu$ is called \emph{balanced}, if there exist $a,b \in (0,\infty)$, such that 
$$a \leq \inf_{x'\in [0,1]^d}\frac{\nu(x+x')}{\nu(x)} \leq \sup_{x'\in [0,1]^d}\frac{\nu(x+x')}{\nu(x)} \leq b \qquad (\forall x\in \Rd).$$
We say that $\nu$ satisfies the \emph{GRS-condition} (Gelfand-Raikov-Shilov condition \cite{GRS64}), if
$$\lim_{n\rightarrow \infty} (\nu(nz))^{\frac{1}{n}} = 1 \qquad (\forall z\in \mathbb{R}^{d}).$$
Note that any submultiplicative weight $\nu$ is balanced and $\nu$-moderate and that any submultiplicative and symmetric weight $\nu$ satisfies $\nu \geq 1$ on $\mathbb{R}^d$.

Typical examples of weights are the functions $\nu(x) = e^{\alpha \vert x \vert^{\beta}}(1+\vert x\vert)^s$. As pointed out in \cite[Example 5.4]{zbMATH06949773}, these weights $\nu$ are balanced whenever $0\leq \beta \leq 1$ and submultiplicative if additionally $\alpha,s \geq 0$. Under these conditions, $\nu$ satisfies the GRS-condition if and only if $0\leq \beta <1$. More general weight functions are given by mixtures of the form 
$$\nu(x) = e^{\alpha \vert x \vert^{\beta}}(1+\vert x \vert)^s (\log(e+\vert x\vert))^t).$$
For $s\geq 0$, $\alpha >0$, $\beta \in (0,1)$ and $t\geq 0$, these weights are examples of \emph{admissible weights}, as defined below, and will play an important role in this article.
\begin{definition}\label{charlyadmissibledef}
\cite{groelei06} A weight function $\nu$ on $\mathbb{R}^d$ is called \emph{admissible}, if 
\begin{itemize}
    \item[(a)] $\nu$ is of the form
\begin{equation}\label{admissible}
\nu(x) = e^{\rho(\Vert x \Vert)} \qquad (x\in \mathbb{R}^d),
\end{equation}
where $\rho:[0,\infty) \longrightarrow [0,\infty)$ is a continuous and concave function with $\rho(0) = 0$ and $\Vert \, . \, \Vert$ any norm on $\mathbb{R}^d$,
    \item[(b)] $\nu$ satisfies the GRS-condition.
\end{itemize}
\end{definition}

The assumptions on $\rho$ imply that $\rho$ is subadditive. Together with the properties of a norm, we immediately see that admissible weights are submultiplicative, balanced, symmetric and satisfy $\nu(0)=1$. In particular, $\nu \geq 1$ on $\mathbb{R}^d$. 

We refer the reader to the article \cite{gr06weightsinTFA} for a comprehensive overview on various classes of weight functions and their relevance in Banach algebra theory. 

\subsection{Relatively separated index sets}\label{Relatively separated index sets}

A countable set $X\subset \Rd$ is called \emph{relatively separated} if
\begin{equation}\label{relativelyseparated}
\sup_{x\in \Rd} \vert X \cap (x+[0,1]^d) \vert < \infty.
\end{equation}

We collect the following preparatory result for later reference.

\begin{lemma}\label{separatedlemma}
    Let $X\subset \Rd$ be a relatively separated set. 
\begin{itemize}
    \item[(a)] \cite[Lemma 1]{gr04-1} For any $s>d$, there exists a constant $C=C(s)>0$ such that 
    $$\sup_{x\in \Rd} \sum_{k\in X} (1+\vert x-k\vert)^{-s} = C <\infty$$
    \item[(b)] \cite[Lemma 2 (a)]{gr04-1} For any $s>d$, there exists a constant $C=C(s)>0$ such that 
    $$\sum_{n\in X} (1+\vert k-n\vert)^{-s}(1+\vert l-n\vert)^{-s} \leq C (1+\vert k-l \vert)^{-s} \qquad (\forall k,l\in X).$$
\end{itemize}
\end{lemma}

\subsection{Bochner sequence spaces}

Occasionally, our analysis will rely on properties on Banach space-valued $\ell^p$-spaces, defined as follows. If $B$ is a Banach space, $X$ a countable index set, and $\nu = (\nu_k)_{k\in X}$ a family of positive numbers (e.g. samples of a weight function as in Subsection \ref{Weight functions}), then for each $p\in [1,\infty]$ the Bochner sequence space $\ell^p(X; B)$ \cite[Chapter 1]{HyNeVeWe16} is defined by 
$$\ell_{\nu}^p(X; B) := \left\lbrace (Y_k)_{k\in X}: Y_k \in B \,\,  (\forall k\in X), (\Vert Y_k \Vert_B \cdot \nu_k)_{k\in X} \in \ell^p(X) \right\rbrace .$$
In case $\nu_k = 1$ for all $k\in X$, we write $\ell_{\nu}^p(X; B) = \ell^p(X; B)$. 

The Bochner sequence spaces $\ell_{\nu}^p(X; B)$ share many properties with the classical $\ell^p$-spaces. In particular (see \cite[Chapter 1]{HyNeVeWe16}), $\ell_{\nu}^p(X; B)$ is a Banach space with respect to the norm 
$$\Vert (Y_k)_{k\in X} \Vert_{\ell_{\nu}^p(X; B)} = \Vert (\Vert Y_k \Vert_B \cdot \nu_k)_{k\in X}\Vert_{\ell^p(X)}$$
for every $1\leq p \leq \infty$. If $1\leq p <\infty$ and $\frac{1}{p}+\frac{1}{q} = 1$, then the dual space $(\ell_{\nu}^p(X; B))^*$ of $\ell_{\nu}^p(X; B)$ is isometrically isomorphic to $\ell_{1/\nu}^q(X; B^*)$. Hence, we identify the latter with $(\ell_{\nu}^p(X; B))^*$ from now on. Moreover, Riesz-Thorin interpolation holds. For later reference, we state the following special case of \cite[Theorem 2.2.1]{HyNeVeWe16}.

\begin{proposition}[Riesz-Thorin interpolation]\label{Riesz-Thorin}
Suppose that $A\in \mathcal{B}(\ell_{\nu}^1(X; B))$ and $A\in \mathcal{B}(\ell_{\nu}^{\infty}(X; B))$. Then $A\in \mathcal{B}(\ell_{\nu}^p(X; B))$ for every $1 < p < \infty$ and 
$$\Vert A \Vert_{\mathcal{B}(\ell_{\nu}^p(X; B))} \leq \max \lbrace \Vert A \Vert_{\mathcal{B}(\ell_{\nu}^1(X; B))}, \Vert A \Vert_{\mathcal{B}(\ell_{\nu}^{\infty}(X; B))} \rbrace \qquad (\forall 1 \leq p \leq \infty).$$
\end{proposition}

\subsection{The Banach algebra \text{$\mathcal{B}(\ell^2(X; \Hil))$}}\label{Operator valued matrices}

Of special interest to us in this article is the Hilbert space $\ell^2(X; \Hil)$ and the Banach algebra $\B(\ell^2(X; \Hil))$ of bounded operators acting on it.
%

Assume that for each $k,l\in X$ we are given a bounded operator $A_{k,l} \in \B(\Hil)$. Then   
\begin{equation}\label{def-Of}
A( f_l )_{l\in X} := \Big( \sum_{l \in X} A_{k,l} f_l \Big)_{k\in X}
\end{equation}
defines a linear operator $A: \D(A) \subseteq \ell^2(X; \Hil) \longrightarrow \ell^2(X; \Hil)$, where 
\begin{equation}\label{domnew}
    \D(A) = \left\lbrace (f_l)_{l\in X} \in \ell^2(X; \Hil) : \sum_{k \in X} \bigg\Vert \sum_{l \in X} A_{k,l} f_l \bigg\Vert^2 < \infty \right\rbrace .
\end{equation}
If we view elements $(f_l)_{l\in X}$ from $\ell^2(X; \Hil)$ as (possibly infinite) column vectors, then we can represent the operator $A$ by the matrix
\begin{equation}\label{matrix}
A = \begin{bmatrix} 
    \ddots & \vdots & \vdots & \, \\
    \dots & A_{k,l} & A_{k,l+1} & \dots \\
    \dots & A_{k+1,l} & A_{k+1, l+1} & \dots \\
    \, & \vdots & \vdots & \ddots & \\ 
    \end{bmatrix} ,
\end{equation}
because (\ref{def-Of}) precisely corresponds to the formal matrix multiplication of $A$ with $(f_l)_{l\in X}$. 

In fact, one can show that there is a one-to-one correspondence between bounded operators $A\in \mathcal{B}(\ell^2(X; \Hil))$ and $\B(\Hil)$-valued matrices $[A_{k,l}]_{k,l\in X}$ which satisfy a certain domain condition. This fact was mentioned in \cite{Maddox:101881}, for the proof details we refer to \cite{kohl21,koebacasheihomosha23}. More precisely:

\begin{proposition}\label{LemmaBounded2}\cite{Maddox:101881,kohl21,koebacasheihomosha23}
For every bounded operator $A \in \mathcal{B}(\ell^2(X; \Hil))$ there exists a unique $\B(\Hil)$-valued matrix $[A_{k,l}]_{k,l \in X}$ such that the action of $A$ on any $(f_l)_{l \in X} \in \ell^2(X; \Hil)$ is given by (\ref{def-Of}). Conversely, if $A = [A_{k,l}]_{k,l\in X}$ is $\B(\Hil)$-valued matrix satisfying $\D(A) = \ell^2(X; \Hil)$, where $\D(A)$ is defined as in (\ref{domnew}), then $A\in \mathcal{B}(\ell^2(X; \Hil))$. 
\end{proposition}
The above proposition motivates the notation 
$$\mathbb{M}(A) = [A_{k,l}]_{k,l\in X}$$
and we call $\mathbb{M}(A)$ the \emph{canonical
matrix representation} of $A$ \cite{koebacasheihomosha23}. 

\begin{remark}\label{involutioncomposition}
Composing bounded operators $A$ and $B$ from $\B(\ell^2(X; \Hil))$ corresponds to matrix multiplication of their corresponding canonical matrix representations. In other words, if $A, B \in \mathcal{B}(\ell^2(X; \Hil))$, then $\mathbb{M}(AB) = \mathbb{M}(A)\cdot \mathbb{M}(B)$, where each entry $[AB]_{k,l} = \sum_{n\in X} A_{k,n} B_{n,l}$ of $\mathbb{M}(AB)$ converges in the operator norm topology with respect to $\B(\Hil)$ \cite{Maddox:101881,kohl21}. In particular, this means that $\mathcal{B}(\ell^2(X; \Hil))$ forms a Banach algebra of $\B(\Hil)$-valued matrices.
\end{remark}

The result below on the canonical matrix representation of the adjoint $A^*$ of $A\in \B(\ell^2(X; \Hil))$ was mentioned in \cite{Maddox:101881}; its proof details can be found in \cite{kohl21}.

\begin{proposition}\label{adjointmatrix}\cite{Maddox:101881,kohl21}
Let $A \in \mathcal{B}(\ell^2(X; \Hil))$ and $\mathbb{M}(A) = [A_{k,l}]_{k,l\in X}$ be its canonical matrix representation. Then the canonical matrix representation of the adjoint operator $A^* \in \mathcal{B}(\ell^2(X; \Hil))$ is given by the relation  
\begin{equation}\label{involution}
\mathbb{M}(A^*) = [A_{k,l}^*]_{k,l\in X}^t
\end{equation}
where the exponent $t$ denotes transposition.
\end{proposition}


\section{Weighted Schur-type algebras}\label{The Schur algebra}

Let $X\subseteq \mathbb{R}^d$ be a relatively separated and $\nu$ be a weight. Then for $1\leq p < \infty$ we define $\MS_{\nu}^p = \MS_{\nu}^p (X)$ to be the space of all $\B(\Hil)$-valued matrices $A=[A_{k,l}]_{k,l\in X}$ for  which  
\begin{equation}\label{Schurnorm}
    \max \left\lbrace \sup_{k\in X} \left( \sum_{l\in X} \Vert A_{k,l} \Vert^p \nu(k-l)^p \right)^{\frac{1}{p}}, \sup_{l\in X} \left( \sum_{k\in X} \Vert A_{k,l} \Vert^p\nu(k-l)^p \right)^{\frac{1}{p}} \right\rbrace
\end{equation}
is finite. For $\nu \equiv 1$ we abbreviate $\MS^p := \MS_{1}^p$. 

It is easy to see that (\ref{Schurnorm}) defines a norm on $\MS_{\nu}^p$. In fact, $\MS_{\nu}^p$ is complete with respect to this norm, as shown below. 

\begin{proposition}\label{completenessproof}
For every $1\leq p < \infty$, $(\MS_{\nu}^p, \Vert \, . \, \Vert_{\MS_{\nu}^p})$ is a Banach space.
\end{proposition}

\begin{proof}
Let $\lbrace A^{(n)} \rbrace_{n=1}^{\infty}$ be a Cauchy sequence in $\MS_{\nu}^p$, i.e. for any given $\varepsilon > 0$, there exists $N = N(\varepsilon) \in \mathbb{N}$ such that for all $m,n \geq N$ 
\begin{equation}\label{Cauchy1}
    \sup_{k\in X} \left( \sum_{l\in X} \Vert A^{(m)}_{k,l}-A^{(n)}_{k,l} \Vert^p \nu(k-l)^p \right)^{\frac{1}{p}} < \varepsilon
\end{equation}
and
\begin{equation}\label{Cauchy2}
\sup_{l\in X} \left( \sum_{k\in X} \Vert A^{(m)}_{k,l}-A^{(n)}_{k,l} \Vert^p \nu(k-l)^p \right)^{\frac{1}{p}} < \varepsilon .
\end{equation}

\emph{Step 1.} Relations (\ref{Cauchy1}) and (\ref{Cauchy2}) imply that for each index pair $(k,l)\in X\times X$, $\lbrace A^{(n)}_{k,l} \rbrace_{n=1}^{\infty}$ is a Cauchy sequence in $\B(\Hil)$ and thus converges to some $A_{k,l} \in \B(\Hil)$. Set $A =  [A_{k,l}]_{k,l\in X}$. We will show that $A^{(n)} \rightarrow A$ in $\MS_{\nu}^p$ and that $A\in \MS_{\nu}^p$.

\emph{Step 2.} For each $n\in \mathbb{N}$, set 
\begin{flalign}
    &\tilde{A}^{(n)} := \big(\tilde{A}^{(n)}_k \big)_{k\in X}, \quad \text{where} \quad \tilde{A}^{(n)}_k = \big(A^{(n)}_{k,l} \nu(k-l)\big)_{l\in X} \in \ell^p(X; \B(\Hil)) \notag \\
    &\tilde{\tilde{A}}^{(n)} := \big(\tilde{\tilde{A}}^{(n)}_l \big)_{l\in X}, \quad \text{where} \quad \tilde{\tilde{A}}^{(n)}_l = \big(A^{(n)}_{k,l} \nu(k-l)\big)_{k\in X} \in \ell^p(X; \B(\Hil)). \notag
\end{flalign}
Then, relation (\ref{Cauchy1}) implies that $\lbrace \tilde{A}^{(n)}\rbrace_{n=1}^{\infty}$ is a Cauchy sequence in the Banach space $\ell^{\infty}(X; \ell^p(X;\B(\Hil)))$ and thus converges to some $\tilde{A} = \big(\tilde{A}_k \big)_{k\in X} \in \ell^{\infty}(X; \ell^p(X;\B(\Hil)))$. Analogously, (\ref{Cauchy2}) implies that $\lbrace \tilde{\tilde{A}}^{(n)}\rbrace_{n=1}^{\infty}$ is a Cauchy sequence in  $\ell^{\infty}(X; \ell^p(X;\B(\Hil)))$ and thus converges to some $\tilde{\tilde{A}} = \big(\tilde{\tilde{A}}_l \big)_{l\in X} \in \ell^{\infty}(X; \ell^p(X;\B(\Hil)))$.

\emph{Step 3.} We denote 
\begin{flalign}
&\tilde{A}_k = \big(\tilde{A}_{k,l} \nu(k-l)\big)_{l\in X} \quad \text{for each } k\in X \notag \\
&\tilde{\tilde{A}}_l = \big(\tilde{\tilde{A}}_{k,l} \nu(k-l)\big)_{k\in X} \quad \text{for each } l\in X. \notag 
\end{flalign}
Then, since $\tilde{A}^{(n)} \rightarrow \tilde{A}$ in $\ell^{\infty}(X; \ell^p(X;\B(\Hil)))$, we obtain that $\tilde{A}^{(n)}_k = \big(A^{(n)}_{k,l} \nu(k-l)\big)_{l\in X} \rightarrow \tilde{A}_k = \big(\tilde{A}_{k,l} \nu(k-l)\big)_{l\in X}$ in $\ell^p(X; \B(\Hil))$ for every $k\in X$. Similarly, we see that $\tilde{\tilde{A}}^{(n)}_l = \big(A^{(n)}_{k,l} \nu(k-l)\big)_{k\in X} \rightarrow \tilde{\tilde{A}}_l = \big(\tilde{\tilde{A}}_{k,l} \nu(k-l)\big)_{k\in X}$ in $\ell^p(X; \B(\Hil))$ for every $l\in X$. The latter two conclusions imply that $[\tilde{A}^{(n)}_k]_l = A^{(n)}_{k,l} \nu(k-l) \rightarrow [\tilde{A}_k]_l = \tilde{A}_{k,l} \nu(k-l)$ in $\B(\Hil)$ and that $[\tilde{\tilde{A}}^{(n)}_l]_k = A^{(n)}_{k,l} \nu(k-l) \rightarrow [\tilde{\tilde{A}}_l]_k = \tilde{\tilde{A}}_{k,l} \nu(k-l)$ in $\B(\Hil)$ for each $(k,l) \in X \times X$. However, by Step 1 we have $A^{(n)}_{k,l} \rightarrow A_{k,l}$ in $\B(\Hil)$, hence we must have  $\tilde{A}_{k,l} = \tilde{\tilde{A}}_{k,l} = A_{k,l}$ for all $k,l\in X$. In particular, this implies that $\tilde{A}_k = (A_{k,l} \nu(k-l))_{l\in X}$ for each $k\in X$, and $\tilde{\tilde{A}}_l = (A_{k,l} \nu(k-l))_{k\in X}$ for each $l\in X$. 

\emph{Step 4.} Let $\varepsilon' >0$ be arbitrary. By Step 2, $\tilde{A}^{(n)} \rightarrow \tilde{A}$ in $\ell^{\infty}(X; \ell^p(X;\B(\Hil)))$. Combined with the observations made in Step 3, we know that there exists some $N' = N'(\varepsilon') \in \mathbb{N}$ such that for all $n\geq N'$
\begin{equation}\label{convergence1}
    \sup_{k\in X} \left( \sum_{l\in X} \Vert A^{(n)}_{k,l}-A_{k,l} \Vert^p \nu(k-l)^p \right)^{\frac{1}{p}} < \varepsilon' .
\end{equation}
Similarly, $\tilde{\tilde{A}}^{(n)} \rightarrow \tilde{A}$ in $\ell^{\infty}(X; \ell^p(X;\B(\Hil)))$ implies that there exists some $N'' = N''(\varepsilon') \in \mathbb{N}$ such that for all $n\geq N''$
\begin{equation}\label{convergence2}
\sup_{l\in X} \left( \sum_{k\in X} \Vert A^{(n)}_{k,l}-A_{k,l} \Vert^p \nu(k-l)^p \right)^{\frac{1}{p}} < \varepsilon' .
\end{equation}
Combining (\ref{convergence1}) with (\ref{convergence2}) yields that for all $n\geq \max \lbrace N', N'' \rbrace$ 
\begin{equation}\label{convergence3}
    \Vert A^{(n)} - A \Vert_{\MS_{\nu}^p} < \varepsilon' .
\end{equation}
Hence $\lbrace A^{(n)} \rbrace_{n=1}^{\infty}$ converges in $\MS_{\nu}^p$ to $A$.

\emph{Step 5.} By relation (\ref{convergence3}) we know that the $\B(\Hil)$-valued matrix $A - A^{(n)}$ belongs to $\MS_{\nu}^p$ for sufficiently large $n$. This implies that $A = (A - A^{(n)}) + A^{(n)} \in \MS_{\nu}^p$ and the proof is complete.
\end{proof}

\begin{remark}
The definition of $\MS_{\nu}^p$ canonically extends to the case $p=\infty$. In fact, we set $\MS_{\nu}^{\infty} = \ell_u^{\infty}(X\times X; \B(\Hil))$, where $u(k,l) = \nu(k-l)$. In this case the completeness of $\MS_{\nu}^{\infty}$ follows immediately from the completeness of Bochner sequence spaces. In Section \ref{Off-diagonal decay conditions via admissible weights} these spaces will be denoted by $\J_{\nu}$ and are discussed in more detail. In the case $\nu(x) = \nu_s(x) = (1+\vert x\vert)^s$ ($x\in \mathbb{R}^d$) of a polynomial weight, $\MS_{\nu_s}^{\infty} = \mathcal{J}_{\nu_s}$ equals the \emph{Jaffard algebra} (see Section \ref{Off-diagonal decay conditions via admissible weights}), whence the notation $\J_{\nu}$.
\end{remark}

We will now focus on the case $p=1$. In fact, for admissible weights $\nu$ (in the sense of Definiton \ref{charlyadmissibledef}), we will hereinafter prove step by step that $\MS_{\nu}^1 \subseteq \B(\ell^2(X;\Hil))$ is a Wiener pair.

First, we show a generalization of Schur's test \cite{gr01} to $\B(\Hil)$-valued matrices.

\begin{lemma}\label{Schurimpliesbounded}
Let $\nu$ be an admissible weight and $m$ be $\nu$-moderate. Then every $A\in \MS_{\nu}^1$ defines (via (\ref{def-Of})) a bounded operator on $\ell_m^p(X;\Hil)$ for all $1\leq p \leq \infty$ and it holds
$$\Vert A \Vert_{\B(\ell_m^p(X;\Hil))} \leq C \Vert A \Vert_{\MS_{\nu}^1} \qquad (\forall 1\leq p \leq \infty),$$
where $C$ is the constant arising from $m$ being $\nu$-moderate.
\end{lemma}

\begin{proof}
Let $A\in \MS_{\nu}^1$ be arbitrary. Then
\begin{flalign}\label{opnorm1}
  \Vert A \Vert_{\B(\ell_m^1(X;\Hil))} &= \sup_{\Vert f \Vert_{\ell_m^1(X;\Hil)} = 1} \sum_{k\in X} \left\Vert \sum_{l\in X} A_{k,l} f_l \right\Vert m(l+k-l) \notag \\
  &\leq C \sup_{\Vert f \Vert_{\ell_m^1(X;\Hil)} = 1} \sum_{k\in X} \sum_{l\in X} \Vert A_{k,l} \Vert \Vert f_l \Vert m(l) \nu(k-l) \notag \\
  &= C\sup_{\Vert f \Vert_{\ell_m^1(X;\Hil)} = 1} \sum_{l\in X} \Vert f_l \Vert m(l) \sum_{k\in X} \Vert A_{k,l} \Vert \nu(k-l) \notag \\
  &\leq C \Vert A \Vert_{\MS_{\nu}^1}.
\end{flalign}
Similarly, 
\begin{flalign}\label{opnorm2}
  \Vert A \Vert_{\B(\ell_m^{\infty}(X;\Hil))} &= \sup_{\Vert f \Vert_{\ell_m^{\infty}(X;\Hil)} = 1} \sup_{k\in X} \left\Vert \sum_{l\in X} A_{k,l} f_l \right\Vert m(l+k-l)\notag \\
  &\leq C \sup_{\Vert f \Vert_{\ell_m^{\infty}(X;\Hil)} = 1} \sup_{k\in X} \sum_{l\in X} \Vert A_{k,l} \Vert \Vert f_l \Vert m(l) \nu(k-l) \notag \\
  &\leq C \Vert A \Vert_{\MS_{\nu}^1}.
\end{flalign}
By Riesz-Thorin interpolation (Proposition \ref{Riesz-Thorin}), the statement follows.
\end{proof}

\begin{remark}
In the scalar-valued case, the reverse inequalities of (\ref{opnorm1}) and (\ref{opnorm2}) hold true in case $m=\nu \equiv 1$ (and hence $C=1$). Consequently, in the scalar-valued setting the norm $\Vert A \Vert_{\MS^1}$ is simply the larger of the operator norms of $A$ on $\ell^1(X)$ and $\ell^{\infty}(X)$ respectively, which yields a shortcut to showing the completeness of $\MS_{\nu}^1$ in this case. However, in the $\B(\Hil)$-valued setting this shortcut is no longer justifiable and proving the completeness of $\MS_{\nu}^1$ is significantly more involved (see Proposition \ref{completenessproof}). 
\end{remark}

By Lemma \ref{Schurimpliesbounded}, $\MS_{\nu}^1$ is continuously embedded into $\B(\ell^2(X;\Hil))$. Thus, by the matrix calculus from Subsection \ref{Operator valued matrices}, there is a canonical way to define a multiplication and an involution on $\MS_{\nu}^1$. In fact, the following holds.

\begin{proposition}\label{Schuralgebraproposition}
If $\nu$ is an admissible weight then $\MS_{\nu}^1$ is a unital Banach *-algebra with respect to matrix multiplication and involution as defined in (\ref{involution}).
\end{proposition}

\begin{proof}
Let $A,B \in \MS_{\nu}^1$. Then, for arbitrary $k\in X$ we have by the submultiplicativity of $\nu$ that
\begin{flalign}
\sum_{l\in X} \Vert [AB]_{k,l} \Vert v(k-l) &\leq \sum_{l\in X} \sum_{n\in X} \Vert A_{k,n} \Vert \Vert B_{n,l} \Vert v(k-n) v(n-l) \notag \\
&= \sum_{n\in X} \Vert A_{k,n} \Vert v(k-n) \sum_{l\in X}  \Vert B_{n,l} \Vert  v(n-l) \notag \\
&\leq \Vert A \Vert_{\MS_{\nu}^1} \Vert B \Vert_{\MS_{\nu}^1}. \notag
\end{flalign}
Consequently, 
$$\sup_{k\in X} \sum_{l\in X} \Vert [AB]_{k,l} \Vert v(k-l) \leq \Vert A \Vert_{\MS_{\nu}^1} \Vert B \Vert_{\MS_{\nu}^1}.$$
Similarly, we see that 
$$\sup_{l\in X} \sum_{k\in X} \Vert [AB]_{k,l} \Vert v(k-l) \leq \Vert A \Vert_{\MS_{\nu}^1} \Vert B \Vert_{\MS_{\nu}^1},$$
hence 
$$\Vert AB \Vert_{\MS_{\nu}^1}  \leq \Vert A \Vert_{\MS_{\nu}^1} \Vert B \Vert_{\MS_{\nu}^1}.$$
Thus, together with Proposition \ref{completenessproof}, we conclude that $\MS_{\nu}^1$ is a Banach algebra. Next, since $\nu$ is a symmetric weight, the definition of the norm (\ref{Schurnorm}) implies that $\Vert A^* \Vert_{\MS_{\nu}^1} = \Vert A \Vert_{\MS_{\nu}^1}$ for all $A\in \MS_{\nu}^1$ (where the involution is defined as in (\ref{involution})). Moreover, since $\nu(0) = 1$, the identity $\mathcal{I}_{\MS_{\nu}^1} = \mathcal{I}_{\B(\ell^2(X;\Hil))} = \text{diag}[\mathcal{I}_{\B(\Hil)}]_{k\in X}$ is contained in $\MS_{\nu}^1$. 
\end{proof}

If $\nu$ is an admissible weight and if $\nu$ additionally satisfies a certain growth condition, then we can show that $\MS_{\nu}^1 \subseteq \B(\ell^2(X;\Hil))$ forms a Wiener pair. Hulanicki's lemma will be the key tool to verify the inverse-closedness of $\MS_{\nu}^1$ in $\B(\ell^2(X;\Hil))$. In order to derive the desired inequality of spectral radii, we follow ideas from \cite{barnes87} and \cite{groelei06}.

We start with a vector-valued version of part of \cite[Lemma 4.6]{barnes87}. The techniques of the proof are basically the same as in \cite{barnes87}. Nevertheless, we provide the reader with the details adapted to our setting. 

\begin{lemma}\label{lemma46}
Let $\nu (x) = \nu_{\delta}(x) = (1+\vert x \vert)^{\delta}$ for some $\delta \in (0,1]$. Then 
$$r_{\MS^1_{\nu}}(A) = r_{\MS^1}(A) \qquad  (\forall A\in \MS^1_{\nu}).$$
\end{lemma}

\begin{proof}
For $0<\varepsilon \leq 1$, let $\mu_{\varepsilon}(x) = (1+\varepsilon \vert x \vert)^{\delta}$. Then $\mu_{\varepsilon}$ is a submultiplicative weight. Combining Proposition \ref{completenessproof} with the first part of the proof of Proposition \ref{Schuralgebraproposition} (where only the submultiplicativity of the corresponding weight is used) yields that $\MS^1_{\mu_{\varepsilon}}$ is a Banach algebra. We also note that  
$$\mu_{\varepsilon} \leq \nu \leq \varepsilon^{-\delta} \mu_{\varepsilon} .$$
This yields $\Vert A \Vert_{\MS^1_{\nu}} \leq \varepsilon^{-\delta} \Vert A \Vert_{\MS^1_{\mu_{\varepsilon}}}$ and consequently 
$$\Vert A^n \Vert_{\MS^1_{\nu}}^{\frac{1}{n}} \leq (\varepsilon^{-\delta})^{\frac{1}{n}} \Vert A^n \Vert_{\MS^1_{\mu_{\varepsilon}}}^{\frac{1}{n}} \qquad (\forall A\in \MS^1_{\nu}).$$
Using Gelfand's formula this implies that
\begin{equation}\label{spectralradius2}
r_{\MS^1_{\nu}}(A) \leq r_{\MS^1_{\mu_{\varepsilon}}}(A) \leq \Vert A \Vert_{\MS^1_{\mu_{\varepsilon}}} \qquad (\forall A\in \MS^1_{\nu}).
\end{equation}
Next, let $A\in \MS^1_{\nu}$. Then, since $t^{\delta} \leq (1+t)^{\delta} \leq 1+t^{\delta}$ for $t\geq 0$, we have
\begin{flalign}
    \sup_{k\in X} \sum_{l\in X} \Vert A_{k,l} \Vert &\leq \sup_{k\in X} \sum_{l\in X} \Vert A_{k,l} \Vert (1+\varepsilon \vert k-l \vert)^{\delta} \notag \\
     &\leq \sup_{k\in X} \sum_{l\in X} \Vert A_{k,l} \Vert (1+\varepsilon^{\delta} \vert k-l \vert^{\delta}) \notag \\
     &\leq \sup_{k\in X} \sum_{l\in X} \Vert A_{k,l} \Vert + \varepsilon^{\delta} \sup_{k\in X} \sum_{l\in X} \Vert A_{k,l}\Vert (1+\vert k-l \vert)^{\delta}. \notag 
\end{flalign}
Interchanging the roles of $k$ and $l$, we analogously obtain
\begin{flalign}
    \sup_{l\in X} \sum_{k\in X} \Vert A_{k,l} \Vert &\leq \sup_{l\in X} \sum_{k\in X} \Vert A_{k,l} \Vert (1+\varepsilon \vert k-l \vert)^{\delta} \notag \\
     &\leq \sup_{l\in X} \sum_{k\in X} \Vert A_{k,l} \Vert + \varepsilon^{\delta} \sup_{l\in X} \sum_{k\in X} \Vert A_{k,l}\Vert (1+\vert k-l \vert)^{\delta}, \notag 
\end{flalign}
and consequently
$$\Vert A \Vert_{\MS^1} \leq \Vert A \Vert_{\MS^1_{\mu_{\varepsilon}}} \leq \Vert A \Vert_{\MS^1} + \varepsilon^{\delta} \Vert A \Vert_{\MS^1_{\nu}} \qquad (\forall A \in \MS^1_{\nu}).$$
Combining the latter with the inequalities (\ref{spectralradius2}) gives us
\begin{equation}\label{inequ1}
    r_{\MS^1_{\nu}}(A) \leq \lim_{\varepsilon \rightarrow 0} \Vert A \Vert_{\MS^1_{\mu_{\varepsilon}}} = \Vert A \Vert_{\MS^1} \qquad (\forall A \in \MS^1_{\nu}).
\end{equation}
Finally, the inequality (\ref{inequ1}) implies that $r_{\MS^1_{\nu}}(A)^n =  r_{\MS^1_{\nu}}(A^n) \leq \Vert A^n \Vert_{\MS^1}$, hence $r_{\MS^1_{\nu}}(A) \leq \Vert A^n \Vert_{\MS^1}^{1/n} \rightarrow r_{\MS^1}(A)$ (as $n\rightarrow \infty$). This implies $r_{\MS^1_{\nu}}(A) \leq r_{\MS^1}(A)$ for all $A\in \MS^1_{\nu}$. On the other hand we clearly have $\MS^1_{\nu} \subseteq \MS^1$, hence $r_{\MS^1_{\nu}}(A) \geq r_{\MS^1}(A)$ for all $A\in \MS^1_{\nu}$ as a consequence of (\ref{spectrum1}). This proves the desired equation.
\end{proof}

Our next result is a vector-valued version of \emph{Barnes' Lemma} from \cite[Lemma 5]{groelei06}. Note that additional estimates to the proof ideas from \cite{barnes87} are necessary in order to prove this result in the $\B(\Hil)$-valued setting.

\begin{theorem}\label{thm47}
Let $\nu_{\delta} (x) = (1+\vert x \vert)^{\delta}$ for some $\delta \in (0,1]$. Then
\begin{equation}\label{mainbarns}
r_{\MS^1_{\nu_{\delta}}} (A) = r_{\MS^1}(A) = r_{\B(\ell^2(X;\Hil))}(A) \qquad (\forall A=A^* \in \MS^1_{\nu_{\delta}}).
\end{equation}
In particular, $\MS^1_{\nu_{\delta}}$ is inverse-closed in $\B(\ell^2(X;\Hil))$ and $\MS^1_{\nu_{\delta}}$ is a symmetric Banach algebra.
\end{theorem}

\begin{proof}
We abbreviate $\B = \B(\ell^2(X;\Hil))$.

By Lemma \ref{lemma46} it suffices to show the equation 
\begin{equation}\label{mainbarns2}
r_{\MS^1} (A) = r_{\B}(A) \qquad (\forall A=A^* \in \MS^1_{\nu_{\delta}}).
\end{equation}
Once we have established (\ref{mainbarns2}), the rest of the statement follows from Hulanicki's lemma (Proposition \ref{Hulanickilemma}). Moreover, it suffices to show (\ref{mainbarns2}) for non-zero $A$, since otherwise the statement is trivial upon an application of Gelfand's formula.

\emph{Step 1.} Fix some non-zero $A=A^* \in \MS^1_{\nu_{\delta}}$.  Since $A$ is self-adjoint it holds
$$\Vert A \Vert_{\MS^1} = \sup_{l\in X} \sum_{k\in X} \Vert A_{k,l} \Vert.$$
Thus, for any arbitrary but fixed $n\in \mathbb{N}$, we have 
\begin{equation}\label{split}
\Vert A^{n+1} \Vert_{\MS^1} \leq \sup_{l\in X} \sum_{\substack{k\in X \\ \vert k-l\vert \leq 2^n} } \Vert [A^{n+1}]_{k,l} \Vert + \sup_{l\in X} \sum_{\substack{k\in X \\ \vert k-l\vert > 2^n} } \Vert [A^{n+1}]_{k,l} \Vert.
\end{equation}
Our strategy is to estimate each of the summands of the right hand side of (\ref{split}) separately, then combine those estimates and finally derive the desired equation (\ref{mainbarns2}) via Gelfand's formula. 

\emph{Step 2.} For each $l\in X$, let $B_{2^n}(l) = \lbrace k\in X: \vert k-l\vert \leq 2^n \rbrace$. We claim that there exists some $C_1>0$ such that
\begin{equation}\label{indexestimate}
    \vert B_{2^n}(l) \vert \leq C_1 2^{nd} \qquad (\forall l\in X).
\end{equation}
To prove the claim, we first note that $B_{2^n}(l) \subseteq B^{\infty}_{2^n}(l):= \lbrace k\in X: \vert k-l \vert_{\infty} \leq 2^n \rbrace$. Since $X \subset \Rd$ is relatively separated, we have that $\gamma := \sup_{x\in \Rd} \vert X \cap (x+[0,1]^d) \vert$ is finite. Furthermore, as $B^{\infty}_{2^n}(l)$ can be covered by $(2\cdot 2^n)^d$ many translated unit cubes, we see that $\vert B^{\infty}_{2^n}(l) \vert \leq (2 \cdot 2^n)^d \gamma$. This proves the claim.

\emph{Step 3.} We estimate the first summand of the right hand side of (\ref{split}). Let $l\in X$ be arbitrary. Then 
\begin{flalign}\label{hürde1}
    \sum_{\substack{k\in X \\ \vert k-l\vert \leq 2^n} } \Vert [A^{n+1}]_{k,l} \Vert &= \sum_{\substack{k\in X \\ \vert k-l\vert \leq 2^n} } \left\Vert \sum_{m\in X} [A^{n}]_{k,m} A_{m,l} \right\Vert \cdot 1 \notag \\
    &\leq \left( \sum_{\substack{k\in X \\ \vert k-l\vert \leq 2^n} } \left\Vert \sum_{m\in X} [A^{n}]_{k,m} A_{m,l} \right\Vert^2 \right)^{\frac{1}{2}} \left( \sum_{\substack{k\in X \\ \vert k-l\vert \leq 2^n} } 1 \right)^{\frac{1}{2}} \notag \\
    &\leq (C_1 2^{nd})^{\frac{1}{2}}\left( \sum_{\substack{k\in X \\ \vert k-l\vert \leq 2^n} } \left\Vert \sum_{m\in X} [A^{n}]_{k,m} A_{m,l} \right\Vert^2 \right)^{\frac{1}{2}}
\end{flalign}
by the Cauchy-Schwarz inequality and (\ref{indexestimate}). Next, since the supremum 
$$\Vert [A^{n+1}]_{k,l}\Vert^2 = \sup_{\Vert f \Vert_{\Hil} = 1} \left\Vert \sum_{m\in X} [A^{n}]_{k,m} A_{m,l} f \right\Vert_{\Hil}^2$$
exists for each $k$, we know from the property of suprema that for each $\varepsilon^{(k)} >0$, there exists some unit norm vector $f^{(k)} \in \Hil$, such that 
$$\left\Vert \sum_{m\in X} [A^{n}]_{k,m} A_{m,l} \right\Vert^2 - \varepsilon^{(k)} < \left\Vert \sum_{m\in X} [A^{n}]_{k,m} A_{m,l} f^{(k)}\right\Vert^2 .$$
Since $A \in \B$ by Lemma \ref{Schurimpliesbounded} and $A\neq 0$ by assumption we can choose $\varepsilon^{(k)} = \Vert A \Vert^{2n}_{\B} (1+\vert k \vert)^{-(d+1)} >0$ and may estimate the sum on the right hand side of (\ref{hürde1}) via
\begin{flalign}\label{hürde2}
    &\sum_{\substack{k\in X \\ \vert k-l\vert \leq 2^n} } \left\Vert \sum_{m\in X} [A^{n}]_{k,m} A_{m,l} \right\Vert^2 \notag \\
    &\leq \Vert A \Vert^{2n}_{\B}\sum_{\substack{k\in X \\ \vert k-l\vert \leq 2^n}} (1+\vert k \vert)^{-(d+1)} + \sum_{\substack{k\in X \\ \vert k-l\vert \leq 2^n} } \left\Vert \sum_{m\in X} [A^{n}]_{k,m} A_{m,l} f^{(k)}\right\Vert^2 \notag \\
    &\leq \Vert A \Vert^{2n}_{\B}\sum_{k\in X} (1+\vert k \vert)^{-(d+1)} + \sum_{\substack{k\in X \\ \vert k-l\vert \leq 2^n}} \sum_{j\in X} \left\Vert \sum_{m\in X} [A^{n}]_{j,m} A_{m,l} f^{(k)}\right\Vert^2 \notag \\
    &\leq C_2 \Vert A \Vert^{2n}_{\B} + \sum_{\substack{k\in X \\ \vert k-l\vert \leq 2^n}} \sum_{j\in X} \left\Vert \sum_{m\in X} [A^{n}]_{j,m} A_{m,l} f^{(k)}\right\Vert^2,
\end{flalign}
where we applied Lemma \ref{separatedlemma} for $s=d+1$ in the last estimate. Now, note that since each $f^{(k)} \in \Hil$ is unit norm and since $A\in \MS^2$ (as a consequence of $\MS^1_{\nu_{\delta}} \subseteq \MS^1 \subseteq \MS^2$) we have by an easy computation that $(A_{m,l} f^{(k)})_{m\in X} \in \ell^2(X;\Hil)$ (for each $l\in X$). Thus, since $A^n \in \B$ by Lemma \ref{Schurimpliesbounded},
\begin{flalign}
    \sum_{j\in X} \left\Vert \sum_{m\in X} [A^{n}]_{j,m} A_{m,l} f^{(k)}\right\Vert^2 &\leq \Vert A^n \Vert_{\B}^2 \Vert (A_{m,l} f^{(k)})_{m\in X} \Vert_{\ell^2(X;\Hil)}^2 \notag \\
    &\leq \Vert A \Vert_{\B}^{2n} \Vert A \Vert_{\MS^2}^2 \notag
\end{flalign} 
holds for each $k$. Combining the latter estimate with (\ref{hürde2}) yields 
\begin{flalign}\label{hürde3}
    \sum_{\substack{k\in X \\ \vert k-l\vert \leq 2^n} } \left\Vert \sum_{m\in X} [A^{n}]_{k,m} A_{m,l} \right\Vert^2 &\leq C_2 \Vert A \Vert^{2n}_{\B} + \sum_{\substack{k\in X \\ \vert k-l\vert \leq 2^n}} \Vert A \Vert_{\B}^{2n} \Vert A \Vert_{\MS^2}^2 \notag \\
    &= C_2 \Vert A \Vert^{2n}_{\B} + \vert B_{2^n}(l) \vert \Vert A \Vert_{\B}^{2n} \Vert A \Vert_{\MS^2}^2 \notag \\
     &\leq C_2 \Vert A \Vert^{2n}_{\B} + C_1 2^{nd} \Vert A \Vert_{\B}^{2n} \Vert A \Vert_{\MS^2}^2, 
\end{flalign}
where we applied (\ref{indexestimate}) in the last estimate. Combining (\ref{hürde3}) with (\ref{hürde1}) and taking the supremum over all $l\in X$ yields
\begin{flalign}\label{firstterm}
\sup_{l\in X}\sum_{\substack{k\in X \\ \vert k-l\vert \leq 2^n} } \Vert [A^{n+1}]_{k,l} \Vert &\leq \left(C_1 C_2 2^{nd} \Vert A \Vert^{2n}_{\B} + C_1^2 4^{nd} \Vert A \Vert_{\B}^{2n} \Vert A \Vert_{\MS^2}^2 \right)^{\frac{1}{2}} \notag \\
&\leq (C_1 C_2)^{\frac{1}{2}} 2^{\frac{nd}{2}} \Vert A \Vert^n_{\B} + C_1 2^{nd} \Vert A \Vert^n_{\B} \Vert A \Vert_{\MS^2}, 
\end{flalign}
where we used the subadditivity of the square root function on $[0,\infty)$.

\emph{Step 4.} We estimate the second summand of the right hand side of (\ref{split}). For $\vert k-l \vert > 2^n$ we have $(1+\vert k-l\vert)^{\delta} > (1+2^n)^{\delta} > 2^{n\delta}$ and consequently $1< 2^{-n\delta} \nu_{\delta}(k-l)$. Therefore, 
\begin{flalign}
    \sup_{l\in X} \sum_{\substack{k\in X \\ \vert k-l\vert > 2^n} } \Vert [A^{n+1}]_{k,l} \Vert &\leq 2^{-n\delta} \sup_{l\in X} \sum_{\substack{k\in X \\ \vert k-l\vert > 2^n} } \Vert [A^{n+1}]_{k,l} \Vert \nu_{\delta}(k-l) \notag \\ 
    &\leq 2^{-n\delta} \Vert A^{n+1} \Vert_{\MS^1_{\nu_{\delta}}}. \notag
\end{flalign}

\emph{Step 5.}
After combining the main estimates from Step 1, Step 3 and Step 4, we obtain that 
\begin{flalign}
\Vert A^{n+1} \Vert_{\MS^1}^{\frac{1}{n+1}} &\leq \Big( (C_1 C_2)^{\frac{1}{2}} 2^{\frac{nd}{2}} \Vert A \Vert^n_{\B} + C_1 2^{nd} \Vert A \Vert^n_{\B} \Vert A \Vert_{\MS^2} + 2^{-n\delta} \Vert A^{n+1} \Vert_{\MS^1_{\nu_{\delta}}} \Big)^{\frac{1}{n+1}} \notag \\
&\leq (C_1 C_2)^{\frac{1}{2(n+1)}} 2^{\frac{nd}{2(n+1)}} \Vert A \Vert^{\frac{n}{n+1}}_{\B} + C_1^{\frac{1}{n+1}} 2^{\frac{nd}{n+1}} \Vert A \Vert_{\B}^{\frac{n}{n+1}} \Vert A \Vert_{\MS^2}^{\frac{1}{n+1}} + 2^{\frac{-n\delta}{n+1}} \Vert A^{n+1} \Vert_{\MS^1_{\nu_{\delta}}}^{\frac{1}{n+1}} ,\notag 
\end{flalign}
where we used the subadditivity of $f(x) = x^{\frac{1}{n+1}}$ on $[0,\infty)$ in the second estimate.
Since $n$ was arbitrary, the latter estimate holds for all $n\in \mathbb{N}$. Letting $n\rightarrow \infty$ yields via Gelfand's formula
$$r_{\MS^1}(A) \leq (2^{d/2} + 2^d) \Vert A \Vert_{\B} + 2^{-\delta} r_{\MS^1_{\nu_{\delta}}}(A).$$
By Lemma \ref{lemma46} we have $r_{\MS^1_{\nu}}(A) = r_{\MS^1}(A)$. Hence the above estimate implies that 
\begin{equation}\label{almostfinal}
r_{\MS^1}(A) \leq (1-2^{-\delta})^{-1}(2^{d/2} + 2^d)\Vert A \Vert_{\B}.
\end{equation}

\emph{Step 6.} Finally, since $A$ was chosen arbitrary, (\ref{almostfinal}) holds for all (non-zero) self-adjoint $A\in \MS^1_{\nu_{\delta}}$. Therefore we have  
$$r_{\MS^1}(A) = r_{\MS^1}(A^n)^{\frac{1}{n}} \leq \left((1-2^{-\delta})^{-1}(2^{d/2} + 2^d))\right)^{\frac{1}{n}} \Vert A^n \Vert_{\B}^{\frac{1}{n}}$$
for all $n\in \mathbb{N}$ and all $A=A^* \in \MS^1_{\nu_{\delta}}$. Letting $n\rightarrow \infty$ yields 
$$r_{\MS^1}(A) \leq r_{\B}(A) \qquad (\forall A=A^* \in \MS^1_{\nu_{\delta}}).$$
Conversely, by (\ref{spectrum1}), $r_{\MS^1}(A) \geq r_{\B}(A)$ holds true, since $\MS_{\nu_{\delta}}^1 \subseteq \MS^1 \subseteq \B$. This completes the proof.
\end{proof}

Following the ideas from \cite{groelei06}, we can extend the above statement to a more general class of weights and obtain a broader class of Schur-type algebras which are inverse-closed in $\B(\ell^2(X;\Hil))$.

We need the following preparatory result.

\begin{lemma}\label{admissiblesequence}
\cite[Lemma 8]{groelei06} For any unbounded admissible weight function $\nu$ there exists a sequence of admissible weights $\nu_n$ with the following properties:
\begin{itemize}
    \item[(a)] $\nu_{n+1} \leq \nu_n \leq \nu$ for all $n\in \mathbb{N}$.
    \item[(b)] There exist $c_n >0$ such that $\nu \leq c_n \nu_n$ for all $n\in \mathbb{N}$.
    \item[(c)] $\lim_{n\rightarrow \infty} \nu_n = 1$ uniformly on compact sets of $\mathbb{R}^d$. 
\end{itemize}
In particular, all the weights $\nu_n$ are equivalent and 
$$r_{\MS^1_{\nu}}(A) = r_{\MS^1_{\nu_n}}(A) \qquad (\forall A\in \MS^1_{\nu}, \forall n\in \mathbb{N}).$$
\end{lemma}

The proof of the next lemma can be adapted almost word by word from the proof of \cite[Lemma 9 (a)]{groelei06}.

\begin{lemma}\label{lemma9}
Let $\nu$ be an admissible weight function satisfying the weak growth condition 
$$\nu(x) \geq C (1+\vert x \vert)^{\delta} \qquad \text{for some } \delta\in (0,1], C>0.$$
Then, with $\nu_n$ is in Lemma \ref{admissiblesequence}, it holds 
$$\lim_{n\rightarrow \infty} \Vert A \Vert_{\MS^1_{\nu_n}} =  \Vert A \Vert_{\MS^1_{\nu}} \qquad (\forall A= A^* \in \MS^1_{\nu}).$$
\end{lemma}

Having the previous two lemmata and Theorem \ref{thm47} at hand, we can prove the following analogue of \cite[Theorem 6, Corollary 7]{groelei06}.

\begin{theorem}\label{Schurspectraltheorem}
Let $\nu$ be an admissible weight satisfying the weak growth condition 
$$\nu(x) \geq C (1+\vert x \vert)^{\delta} \qquad \text{for some } \delta\in (0,1], C>0.$$
Then 
\begin{equation}\label{spectralradiiequ2}
    r_{\MS^1_{\nu}}(A) = r_{\B(\ell^2(X;\Hil))}(A) \qquad (\forall A = A^* \in \MS^1_{\nu}).
\end{equation}
In particular, $\MS^1_{\nu}$ is inverse-closed in $\B(\ell^2(X;\Hil))$ and $\MS^1_{\nu}$ is a symmetric Banach algebra.
\end{theorem}

\begin{proof}
We follow the arguments of \cite[Lemma 9 (b)]{groelei06}.

By the equivalence of the weights $\nu_n$ from Lemma \ref{admissiblesequence} we have
$$r_{\MS^1_{\nu}}(A)^m = r_{\MS^1_{\nu}}(A^m) = r_{\MS^1_{\nu_n}}(A^m) \leq \Vert A^m \Vert_{\MS^1_{\nu_n}} \qquad (\forall A\in \MS^1_{\nu}, \forall m,n\in \mathbb{N}).$$
By Lemma \ref{lemma9}, this implies that 
$$r_{\MS^1_{\nu}}(A)^m \leq \lim_{n\rightarrow \infty} \Vert A^m \Vert_{\MS^1_{\nu_n}} = \Vert A^m \Vert_{\MS^1} \qquad (\forall A=A^*\in \MS^1_{\nu}, \forall m\in \mathbb{N}).$$
Thus, taking $m$-th roots and letting $m\rightarrow \infty$ yields 
$$r_{\MS^1_{\nu}}(A) \leq r_{\MS^1}(A) \qquad (\forall A=A^*\in \MS^1_{\nu}).$$
However, since $\nu$ is weakly growing by assumption, we have $\MS^1_{\nu} \subseteq \MS^1_{\nu_{\delta}}$. Therefore, Theorem \ref{thm47} implies that 
$$r_{\MS^1_{\nu}}(A) \leq r_{\MS^1}(A) = r_{\B(\ell^2(X;\Hil))}(A) \qquad (\forall A=A^*\in \MS^1_{\nu}) .$$
Conversely, by (\ref{spectrum1}), we have $r_{\B(\ell^2(X;\Hil))}(A) \leq r_{\MS^1_{\nu}}(A)$ for all $A=A^*\in \MS^1_{\nu}$ since $\MS^1_{\nu} \subseteq \B(\ell^2(X;\Hil))$. Thus (\ref{spectralradiiequ2}) follows. Now, having established the desired equation of spectral radii, the rest of the theorem follows from Hulanicki's lemma (Proposition \ref{Hulanickilemma}).
\end{proof}

\begin{remark}
1.) The weak growth condition in Theorem \ref{Schurspectraltheorem} is essential. In fact, by giving a concrete counter-example, Tessera proved in \cite{Tess09} that the unweighted scalar-valued Schur algebra $\MS^1(\mathbb{N})$ is \emph{not} inverse-closed in $\B(\ell^p(\mathbb{N};\mathbb{C}))$ for each $1<p<\infty$. As emphasized in the same reference, $\MS^1(\mathbb{N})$ is not even symmetric \cite{jen70}.

2.) Sun proved in \cite{sun07a} the inverse-closedness of the scalar-valued versions of $\MS^p_{\nu}(X)$ in $\B(\ell^2(X;\mathbb{C}))$. We believe that the methods from \cite{sun07a} can be used to obtain analogous results in the $\B(\Hil)$-valued setting. We leave a detailed investigation to future research.
\end{remark}

\section{Polynomial and more general off-diagonal decay conditions}\label{Off-diagonal decay conditions via admissible weights}

After proving that the Schur algebra $\MS_{\nu}^1$ is inverse-closed in $\B(\ell^2(X;\Hil))$, we can follow the ideas from \cite{groelei06} and derive more examples of Wiener pairs $\A \subseteq \B(\ell^2(X;\Hil))$. Most notably, we will consider a $\B(\Hil)$-valued version of the Jaffard algebra and prove its inverse-closedness in $\B(\ell^2(X;\Hil))$.

Let $X\subset \mathbb{R}^d$ be  relatively separated and $\nu$ be an admissible weight in the sense of Definition \ref{charlyadmissibledef}. Then we define $\J_{\nu} = \J_{\nu}(X)$ to be the space of all $\B(\Hil)$-valued matrices $A=[A_{k,l}]_{k,l\in X}$ for which there exists $C>0$ such that
\begin{equation}\label{offdiagonalcondition}
\Vert A \Vert_{\J_{\nu}} := \sup_{k,l\in X} \Vert A_{k,l} \Vert \cdot \nu(k-l) \leq C.  
\end{equation} 
Again, it is easy to see that (\ref{offdiagonalcondition}) indeed defines a norm on $\J_{\nu}$ and that $(\J_{\nu}, \Vert \, . \, \Vert_{\J_{\nu}})$ is precisely the Bochner sequence space $\ell_{u}^{\infty}(X\times X; \B(\Hil))$, where $u(k,l) = \nu(k-l)$, and thus a Banach space. 

The proof of the following lemma is analogous to the proof of \cite[Lemma 1 (b)]{groelei06} and is therefore omitted. 

\begin{lemma}\label{lemmaalgebrageneral}
Let $\nu$ be an admissible weight. If 
\begin{equation}\label{generalv1}
    \sup_{k\in X} \sum_{l\in X} \nu(k-l)^{-1} <\infty 
\end{equation}
then $\J_{\nu}$ is contained in $\MS^1$. If additionally $\exists C>0$ such that 
\begin{equation}\label{generalv2}
\sum_{n\in X} \nu(k-n)^{-1}\nu(n-l)^{-1} < C \nu(k-l) \qquad (\forall k,l\in X) 
\end{equation}
then $\J_{\nu}$ is a unital *-algebra which satisfies $\J_{\nu} \subseteq \B(\ell^2(X;\Hil))$. 
\end{lemma}

\begin{remark}\label{generalvremark}
Let $\nu_s$ be the polynomial weight $\nu_s(x) = (1+\vert x\vert)^s$ on $\Rd$ with decay parameter $s>d$. Then $\nu_s$ is admissible. Furthermore, by Lemma \ref{separatedlemma}, $\nu_s$ satisfies the conditions (\ref{generalv1}) and (\ref{generalv2}). Thus Lemma \ref{lemmaalgebrageneral} can be viewed as a generalization of Lemma \ref{separatedlemma}. Moreover, in case $X=\mathbb{Z}^d$, condition (\ref{generalv1}) is equivalent to $\nu^{-1} \in \ell^1(\mathbb{Z}^d)$ and condition (\ref{generalv2}) reads as $\nu^{-1} \ast \nu^{-1} \leq C \nu^{-1}$ (i.e. $\nu^{-1}$ is \emph{subconvolutive} \cite{Fei79}).
\end{remark}

The norm $\Vert \, . \, \Vert_{\J_{\nu}}$ is not submultiplicative. This obstacle can be circumvented by working with the equivalent norm 
\begin{equation}\label{generalvnorm}
    \vert\Vert A \Vert\vert_{\mathcal{J}_{\nu}} := \sup_{\substack{B\in \mathcal{J}_{\nu} \\ \Vert B \Vert_{\mathcal{J}_{\nu}} = 1}} \Vert A\cdot B \Vert_{\mathcal{J}_{\nu}} \qquad (A \in \mathcal{J}_{\nu}).
\end{equation}

Next, we consider the $\B(\Hil)$-valued version of the Banach algebra $\B_{u,s}$ from \cite{groelei06}. Let $s>0$, $u$ be an admissible weight and $\nu$ be the admissible weight defined by $\nu(x) = u(x)(1+\vert x \vert)^s$. Then we define $\B_{u,s} = \B_{u,s}(X)$ to be the Banach space $\B_{u,s} := \MS_u^1 \cap \J_{\nu}$ equipped with the norm 
\begin{equation}\label{Busnorm}
    \Vert A \Vert_{\B_{u,s}} := 2^s \Vert A \Vert_{\MS_u^1} + \Vert A \Vert_{\J_{\nu}}.
\end{equation}
The constant $2^s$ in (\ref{Busnorm}) stems from the subadditivity of $\nu_s$ on $\mathbb{R}^d$ and plays an important role in the proof of the next lemma.

\begin{lemma}\label{Buslemma}
The Banach space $\B_{u,s}$ is a solid unital Banach *-algebra and $\B_{u,s} \subseteq \MS_u^1 \subseteq \B(\ell^2(X;\Hil))$. Moreover, if $s>d$ then $\J_{\nu} \subseteq \MS_u^1$ and thus $\B_{u,s} = \J_{\nu}$ with equivalent norms.     
\end{lemma}

\begin{proof}
The proof of $\B_{u,s}$ being a Banach *-algebra and the moreover-part is exactly the same as the proof of \cite[Lemma 2]{groelei06} after replacing absolute values with operator norms and applying Lemma \ref{lemmaalgebrageneral} instead of \cite[Lemma 1]{groelei06} (Lemma \ref{separatedlemma} is applied as well). The statement $\B_{u,s} \subseteq \MS_u^1 \subseteq \B(\ell^2(X;\Hil))$ is obviously true.     
\end{proof}

Now we can prove the main theorem of this section. We will establish the proof details (which also appear in \cite{groelei06}) so that the reader can fully grasp the key ideas. 

\begin{theorem}\label{Busthm}
Assume that $u$ is an admissible weight, let $\delta, s >0$, $u\geq \nu_{\delta}$ and set $\nu = u\nu_s$. Then
\begin{equation}\label{Busthmeq}
    r_{\B_{u,s}} (A) = r_{\B(\ell^2(X;\Hil))} (A) \qquad (\forall A=A^* \in \B_{u,s}).
\end{equation}
In particular, $\B_{u,s}$ is inverse-closed in $\B(\ell^2(X;\Hil))$ and $\B_{u,s}$ is a symmetric Banach algebra.
\end{theorem}

\begin{proof}
Once again, Hulanicki's Lemma (Proposition \ref{Hulanickilemma}) is the key for proving spectral invariance. The crucial idea for  deriving the desired inequality of spectral radii is a proof routine called Brandenburg's trick \cite{BRANDENBURG75}.  

By submultiplicativity of $u$ and subadditivity of $\nu_s$, the weight $\nu = u\nu_s$ satisfies the inequality 
$$\nu(x+y) \leq 2^s u(x) u(y) (\nu_s(x) + \nu_s(y)) = 2^s \left( \nu(x)u(y) + u(x)\nu(y) \right)$$
on $\mathbb{R}^d$. This implies that for any $A,B\in \B_{u,s} \subseteq \J_{\nu}$
\begin{flalign}
\Vert AB \Vert_{\J_{\nu}} &= \sup_{k,l\in X} \left\Vert \sum_{n\in X} A_{k,n} B_{n,l} \right\Vert \nu(k-n+n-l) \notag \\
&\leq 2^s \sup_{k,l\in X} \sum_{n\in X} \Vert A_{k,n} \Vert \Vert B_{n,l} \Vert \big( \nu(k-n)u(n-l) + u(k-n)\nu(n-l) \big) \notag \\
&\leq 2^s ( \Vert A \Vert_{\J_{\nu}} \Vert B \Vert_{\MS_{u}^1} + \Vert B \Vert_{\J_{\nu}} \Vert A \Vert_{\MS_{u}^1}). \notag
\end{flalign}
Consequently, 
$$\Vert A^{2n} \Vert_{\J_{\nu}} \leq 2^{s+1} \Vert A^{n} \Vert_{\J_{\nu}} \Vert A^{n} \Vert_{\MS_{u}^1} \qquad (\forall n\in \mathbb{N}).$$
In particular, this yields the estimate
\begin{flalign}
\Vert A^{2n} \Vert_{\B_{u,s}} &= 2^s \Vert A^{2n} \Vert_{\MS_u^1} + \Vert A^{2n} \Vert_{\J_{\nu}} \notag \\
&\leq 2^s \Vert A^{n} \Vert^2_{\MS_u^1} + 2^{s+1} \Vert A^{n} \Vert_{\J_{\nu}} \Vert A^{n} \Vert_{\MS_{u}^1} \notag \\
&\leq 2^{s+1} \Vert A^{n} \Vert_{\MS_u^1} (2^s \Vert A^{n} \Vert_{\MS_u^1} + \Vert A^{n} \Vert_{\J_{\nu}}) \notag \\
&= 2^{s+1} \Vert A^{n} \Vert_{\MS_u^1} \Vert A^{n} \Vert_{\B_{u,s}} \qquad (\forall n\in \mathbb{N}).\notag
\end{flalign}
Taking $2n$-th roots and letting $n\rightarrow \infty$ yields via Gelfand's formula that 
$$r_{\B_{u,s}} (A) \leq \lim_{n\rightarrow \infty} 2^{\frac{s+1}{2n}} \Vert A^{n} \Vert_{\MS_u^1}^{\frac{1}{2n}} \Vert A^{n} \Vert_{\B_{u,s}}^{\frac{1}{2n}} = r_{\MS_u^1}(A)^{\frac{1}{2}}r_{\B_{u,s}}(A)^{\frac{1}{2}}$$
and consequently 
$$r_{\B_{u,s}} (A) \leq r_{\MS_u^1}(A) \qquad (\forall A\in \B_{u,s}).$$
In particular, since $\MS_u^1 \subseteq \MS^1_{\nu_{\delta}}$, we can apply Theorem \ref{thm47} and obtain that
$$r_{\B_{u,s}} (A) \leq r_{\MS_u^1}(A) = r_{\MS^1}(A) =  r_{\B(\ell^2(X;\Hil))} (A) \qquad (\forall A=A^* \in \B_{u,s}).$$
Now, an application of Hulanicki's lemma (Proposition \ref{Hulanickilemma}) yields the desired statement.
\end{proof}

Finally, we consider the polynomial weight $\nu_s$ for $s>d$ and call $\J_s := \J_{\nu_s}$ the \emph{Jaffard algebra}. By Remark \ref{generalvremark}, and Lemma \ref{Buslemma}, the Jaffard algebra is (up to norm-equivalence via (\ref{generalvnorm})) a unital Banach *-algebra whenever $s>d$. By combining Lemma \ref{Buslemma} with Theorem \ref{Busthm}, we obtain the following generalization of Jaffard's Lemma \cite{ja90} to the $\B(\Hil)$-valued setting.   

\begin{corollary}[Jaffard's Lemma]\label{Jaffardinverseclosed}
For every $s>d$, the Jaffard algebra $({\J}_{s}, \vert\Vert \, . \, \Vert\vert_{\J_{s}})$ is inverse-closed in $\B(\ell^2(X;\Hil))$. In particular, $({\J}_{s}, \vert\Vert \, . \, \Vert\vert_{\J_{s}})$ is a symmetric Banach algebra whenever $s>d$.    
\end{corollary}


\section{The Baskakov-Gohberg-Sjöstrand algebra}\label{The Baskakov-Gohberg-Sjöstrand algebra}

Next we consider the $\B(\Hil)$-valued \emph{Baskakov-Gohberg-Sjöstrand algebra} and weighted versions of it \cite{Baskakov1990,Gohberg89,Sjöstrand1994-1995}. This time we work with the relatively separated index set $X=\mathbb{Z}^d$. Note that some of the results appearing in this section have also been stated in \cite{kri11}. Therefore, we will avoid some of the proof details and refer to \cite{kri11} and other related references instead. 

For a weight function $\nu$ on $\mathbb{R}^d$, let $\C_{\nu}$ be the space of $\B(\Hil)$-valued matrices $A=[A_{k,l}]_{k,l\in \mathbb{Z}^d}$ for which
\begin{equation}\label{Cvnorm}
    \Vert A \Vert_{\C_{\nu}}:= \sum_{l\in \mathbb{Z}^d} \sup_{k\in \mathbb{Z}^d} \Vert A_{k,k-l} \Vert \nu(l) <\infty .
\end{equation}
It is immediately clear, that (\ref{Cvnorm}) defines indeed a norm on $\C_{\nu}$. In fact
$$\Vert A \Vert_{\C_{\nu}} = \Big\Vert \Big( \big\Vert (\Vert A_{k,k-l} \Vert)_{k\in \mathbb{Z}^d}\big\Vert_{\ell^{\infty}(\mathbb{Z}^d)}\Big)_{l\in \mathbb{Z}^d} \Big\Vert_{\ell^1_{\nu}(\mathbb{Z}^d)},$$
hence $\C_{\nu}$ is isometrically isomorphic to the Bochner space $\ell^1_{\nu}(\mathbb{Z}^d; \ell^{\infty}(\mathbb{Z}^d; \B(\Hil)))$ and thus a Banach space.

The following notation will be useful for us:
For $A = [A_{k,l}]_{k,l\in \mathbb{Z}^d}\in \C_{\nu}$, let $d_A(l)$ be the supremum of the operator norms of the entries of $A$ along its $l$-th side diagonal, that is, 
\begin{equation}\label{d_A}
    d_A(l) = \sup_{k\in \mathbb{Z}^d} \Vert A_{k,k-l} \Vert \qquad (l\in \mathbb{Z}^d).
\end{equation}
The supremum $d_A(l)$ exists for each $l$ and, in fact, $\Vert d_A\Vert_{\ell_{\nu}^1(\mathbb{Z}^d)} = \Vert A \Vert_{\C_{\nu}}$. Furthermore, as we are in a Hilbert space setting (see also Subsection \ref{Operator valued matrices} and \cite[Theorem 4.16]{Maddox:101881}), $d_A(l)$ is the operator-norm of the $n$-th side-diagonal $D_A(n)$ of $A$, given by 
\begin{equation}\label{sidediagonal}
[D_A(n)]_{k,l} = 
\begin{cases}
A_{k,l} \qquad &\text{ if }l = k-n\\
0 \qquad &\text{ if }l \neq k-n
\end{cases}.
\end{equation}

\begin{proposition}\label{Cvembedding}
Let $m$ be a $\nu$-moderate weight. Then each $A\in \C_{\nu}$ defines a bounded operator on $\B(\ell_m^p(X;\Hil))$ for each $1\leq p \leq \infty$. 
\end{proposition}

\begin{proof}
We follow the proof steps of the exposition \cite{gr10-2}. For $g = (g_l)_{l\in \mathbb{Z}^d} \in \ell_m^p(X;\Hil)$ we have $c=(c_l)_{l\in\mathbb{Z}^d} := (\Vert g_l \Vert)_{l\in \mathbb{Z}^d} \in \ell_m^p(\mathbb{Z}^d)$. With the notation from (\ref{d_A}) we obtain the component-wise estimate   
\begin{flalign}
 \Vert [Ag]_k \Vert &\leq \sum_{l\in \mathbb{Z}^d} \Vert A_{k,l} \Vert \, \Vert g_l \Vert \notag \\
 &\leq \sum_{l\in \mathbb{Z}^d} d_A(k-l) c_l \notag \\
 &= (d_A \ast c) (k) \qquad (\forall k\in \mathbb{Z}^d). \notag
\end{flalign}
Therefore, the claim follows from Young's inequality for weighted $\ell^p$-spaces (see e.g. \cite{gr01}), as we have 
\begin{flalign}
  \Vert Ag \Vert_{\ell_m^p(X;\Hil)} &= \Vert (\Vert [Ag]_k \Vert)_{k\in \mathbb{Z}^d} \Vert_{\ell_m^p(\mathbb{Z}^d)} \notag \\
  &\leq \Vert d_A \ast c \Vert_{\ell_m^p(\mathbb{Z}^d)} \notag \\
  &\leq C \Vert d_A \Vert_{\ell_{\nu}^1(\mathbb{Z}^d)} \Vert c \Vert_{\ell_m^p(\mathbb{Z}^d)} = C \Vert A \Vert_{\C_{\nu}} \Vert g \Vert_{\ell_m^p(X;\Hil)}, \notag  
\end{flalign}
where the constant $C$ stems from the assumption that $m$ is $\nu$-moderate.
\end{proof}

By the above we may call the Baskakov-Gohberg-Sjöstrand algebra the \emph{algebra of convolution-dominated matrices} \cite{gr10-2}.

\begin{proposition}\label{Cvalgebra}\cite{kri11}
For any submultiplicative and symmetric weight $\nu$, $\C_{\nu}$ is a unital Banach *-algebra with respect to matrix multiplication and involution defined as in (\ref{involution}). In particular, $\C_{\nu} \subseteq \B(\ell^2(X;\Hil))$.
\end{proposition}

\begin{proof}
Again we follow the proof steps of the exposition \cite{gr10-2}; the verification of $\C_{\nu}$ being a Banach algebra has already been established in \cite[Lemma 2]{Baskakov1997}.

We have already observed that $\C_{\nu}$ is a Banach space. Let $A,B \in \C_{\nu}$. Then, with the terminology from (\ref{d_A}) we have that
\begin{flalign}
    \Vert [AB]_{k,k-l} \Vert &\leq \sum_{m\in \mathbb{Z}^d} \Vert A_{k,m} \Vert \, \Vert B_{m,k-l} \Vert \notag \\
    &\leq \sum_{m\in \mathbb{Z}^d} d_A(k-m) \, d_B(l-(k-m)) \notag \\
    &= (d_A \ast d_B)(l), \notag
\end{flalign}
which yields the pointwise estimate 
$$d_{AB} (l) \leq (d_A \ast d_B)(l) \qquad (\forall l\in \mathbb{Z}^d).$$
Thus, since $\nu$ is submultiplicative, Young's inequality yields
\begin{flalign}
\Vert AB \Vert_{\C_{\nu}} &= \Vert d_{AB} \Vert_{\ell_{\nu}^1(\mathbb{Z}^d)} \notag \\
&\leq \Vert d_A \ast d_B \Vert_{\ell_{\nu}^1(\mathbb{Z}^d)} \notag \\
&\leq \Vert d_A \Vert_{\ell_{\nu}^1(\mathbb{Z}^d)} \Vert d_{B} \Vert_{\ell_{\nu}^1(\mathbb{Z}^d)}  = \Vert A \Vert_{\C_{\nu}} \Vert B \Vert_{\C_{\nu}}, \notag
\end{flalign}
hence $\C_{\nu}$ is a Banach algebra. The fact that the involution is an isometry follows from the symmetry of $\nu$. Indeed, since 
$$d_{A^*}(l) = \sup_{k\in \mathbb{Z}^d} \Vert [A^*]_{k,k-l} \Vert = \sup_{k\in \mathbb{Z}^d} \Vert (A_{k-l,k})^* \Vert = \sup_{k\in \mathbb{Z}^d} \Vert A_{k,k+l} \Vert = d_A(-l)$$
we see that
\begin{flalign}
\Vert A^* \Vert_{\C_{\nu}} &= \sum_{l\in \mathbb{Z}^d} d_{A^*}(l)v(l) \notag \\
&= \sum_{l\in \mathbb{Z}^d} d_{A}(-l)v(l) \notag \\
&= \sum_{l\in \mathbb{Z}^d} d_{A}(-l)v(-l) = \Vert A \Vert_{\C_{\nu}}. \notag 
\end{flalign}
Since $v(0)=1$, $\Vert \mathcal{I}_{\B(\ell^2(X;\Hil))}\Vert_{\C_{\nu}} =1$, hence $\C_{\nu}$ has an identity. The fact that $\C_{\nu} \subseteq \B(\ell^2(X;\Hil))$ follows from Proposition \ref{Cvembedding}.
\end{proof}

As before, verifying that $\C_{\nu} \subseteq \B(\ell^2(X;\Hil))$ forms a Wiener pair is a subtle task. However, since this has already been done by Baskakov \cite{Baskakov1997} in 1997, we will only outline his key arguments adapted to our setting (compare with \cite{kri11} or the proof outline in \cite{gr10-2} of the scalar-valued case).

At the heart of Baskakov's proof \cite{Baskakov1997} lies a generalization of the classical Wiener's Lemma \cite{Wiener32} to its analogue on absolutely convergent Fourier series with coefficients in a (possibly non-commutative) Banach algebra by Bochner and Phillips \cite[Theorem 1]{BochnerPhillips42}. 

Given $t \in \mathbb{R}^d$, let $M_t \in \B(\ell^2(X;\Hil))$ be the \emph{modulation operator} defined as diagonal matrix 
$$M_t = \text{diag}[e^{2\pi i k\cdot t} \mathcal{I}_{\B(\Hil)}]_{k\in \mathbb{Z}^d}.$$
Then $M_t$ acts on $(g_l)_{l\in \mathbb{Z}^d} \in \ell^2(X;\Hil)$ via $M_t (g_l)_{l\in \mathbb{Z}^d} = (e^{2\pi i l\cdot t} g_l)_{l\in \mathbb{Z}^d}$ and 
$$M:\mathbb{T}^d \longrightarrow \B(\ell^2(X;\Hil)), \quad t\mapsto M_t$$
is a unitary representation of the $d$-dimensional torus $\mathbb{T}^d$, which is the dual group of the abelian group $\mathbb{Z}^d$ $-$ our underlying index set. Then, to each $A = [A_{k,l}]_{k,l\in \mathbb{Z}^d} \in \B(\ell^2(X;\Hil))$, we assign the $\B(\ell^2(X;\Hil))$-valued $1$-periodic function 
$$f_A(t) := M_t A M_{-t} \qquad (t\in \mathbb{R}^d),$$
i.e. we may identify $f_A$ with a $\B(\ell^2(X;\Hil))$-valued function on $\mathbb{T}^d$. Note that the canonical matrix representation of each operator $f_A(t)$ is given by 
\begin{equation}\label{matrixrepr}
    \mathbb{M}(f_A(t)) = [A_{k,l} e^{2\pi i (k-l)\cdot t}]_{k,l\in \mathbb{Z}^d}.
\end{equation}
Now, if $A$ is contained in the unweighted Baskakov-Gohberg-Sjöstrand algebra $\C = \C_1$, then 
\begin{equation}\label{opFourerseries}
    f_A(t) = \sum_{n\in \mathbb{Z}^d} D_A(n) e^{2\pi i n\cdot t}
\end{equation}
admits a Fourier series of operators, where $D_A(n)$ is the $n$-th side-diagonal of $A$ is defined in (\ref{sidediagonal}), and the Fourier series (\ref{opFourerseries}) converges absolutely in the operator norm, since 
\begin{equation}\label{absconv}
\Vert f_A(t) \Vert_{\B(\ell^2(X;\Hil))} \leq \sum_{n\in \mathbb{Z}^d} \Vert D_A(n) \Vert_{\B(\ell^2(X;\Hil))} = \sum_{n\in \mathbb{Z}^d} d_A(n) = \Vert A \Vert_{\C}.
\end{equation}
Note that the expansion (\ref{opFourerseries}) can be deduced directly from the matrix representation (\ref{matrixrepr}) of $f_A$. 
Now assume that $A\in \C$ is invertible in $\B(\ell^2(X;\Hil))$. Then each operator $f_A(t)$ is invertible in $\B(\ell^2(X;\Hil))$ with inverse given by $f_A(t)^{-1} = f_{A^{-1}}(t) = M_t A^{-1} M_{-t}$. Now Bochner's and Phillips' theorem \cite[Theorem 1]{BochnerPhillips42} is applicable and guarantees that $f_{A^{-1}}(t)$ admits an absolutely convergent operator-valued Fourier series for each $t \in \mathbb{T}^d$, whose coefficients are precisely given by the side-diagonals $D_{A^{-1}}(n)$ of $A^{-1}$. Consequently (compare with (\ref{absconv})), $A^{-1} \in \C$ and thus $\C$ is inverse-closed in $\B(\ell^2(X;\Hil))$.

In the weighted case $\C_{\nu}$ the proof idea is essentially the same. For the inverse-closedness of $\C_{\nu}$ in $\B(\ell^2(X;\Hil))$ to hold true, $\nu$ needs to satisfy the GRS-condition, since then the Banach algebra $L_{\nu}(\mathbb{Z}^d)$ as defined in \cite{Baskakov1997} is semi-simple. Instead of the Bochner-Phillips theorem from \cite{BochnerPhillips42}, a more abstractly formulated version, namely \cite[Theorem 1]{Baskakov1997}, is applied. 

We refer the reader to \cite{Baskakov1997} as well as \cite{kri11} for more details; compare also with \cite{Baskakov1990}. 

\begin{theorem}\label{Cvinverseclosed}\cite{Baskakov1997,kri11}
Let $\nu$ be a submultiplicative and symmetric weight satisfying the GRS-condition. Then $\C_{\nu} \subseteq \B(\ell^2(\mathbb{Z}^d;\Hil))$ forms a Wiener pair. In particular, $\C_{\nu}$ is a symmetric Banach algebra.
\end{theorem}

\begin{remark}
In \cite[Theorem 7.2]{baskri14} quantitative estimates for $\Vert A^{-1}\Vert_{\C_{\nu}}$ have been derived for continuous analogs of $\C_{\nu}$. As described in the last paragraph of the introduction, we believe that a discrete analogue of \cite[Theorem 7.2]{baskri14} can be derived via discretization as mentioned in \cite[Remark 5.6]{zbMATH06949773}. Moreover, we expect that methods of \cite{Bas97} could be used to obtain quantitative (i.e., stronger) statements regarding other Wiener pairs discussed so far. We leave a detailed investigation on that matter to future research.
\end{remark}


\section{Anisotropic decay conditions}\label{Anisotropic decay conditions}

In this section we will derive further examples of Wiener pairs from a given Wiener pair $\A \subseteq \B(\ell^2(X;\Hil))$, where we assume $\A$ to be \emph{solid}, meaning that if $A = (A_{k,l})_{k,l \in X} \in \mathcal{A}$ and $B = (B_{k,l})_{k,l\in X}$ is a $\B(\Hil)$-valued matrix such that $\Vert B_{k,l} \Vert \leq \Vert A_{k,l} \Vert$ for all $k,l \in X$, then $B \in \mathcal{A}$ and $\Vert B \Vert_{\mathcal{A}} \leq \Vert A \Vert_{\mathcal{A}}$. This will be done by considering derivations on $\A$ (see Section \ref{Banach algebras}). 
The analysis of this subsection is based on the article \cite{groeklotz10} by Gröchenig and Klotz. Hence we will often only point out the main ideas and refer the reader to \cite{groeklotz10} for the details.

\begin{proposition}\label{d=1derivationthm}\cite[Theorem 3.4]{groeklotz10}
Let $\A$ be a symmetric unital Banach (*-) algebra and $\delta: \D(\delta) \longrightarrow \A$ be a symmetric derivation on $\A$ with $\mathcal{I}_{\A} \in \D(\delta)$. Then $\D(\delta)\subset \A$ is a symmetric unital Banach (*-) algebra and in particular a Wiener pair. Moreover, the quotient rule
$$\delta(A^{-1})=-A^{-1} \delta(A) A^{-1}$$
is valid and yields the explicit norm estimate 
$$\Vert A^{-1} \Vert_{\D(\delta)} \leq \Vert A^{-1} \Vert_{\A}^2 \Vert A \Vert_{\D(\delta)}$$
for all $A\in \D(\delta)$.
\end{proposition}

The above statement can be generalized by iterative means via \emph{commuting derivations} \cite{groeklotz10}. More precisely, we iteratively define compositions of derivations via
$$\delta_{1}\dots \delta_{n}: \D(\delta_{1}\dots \delta_{n})= \D(\delta_{1}, \D( \delta_{2} \dots \delta_{n})) \longrightarrow \D( \delta_{2} \dots \delta_{n}) \subseteq \A, $$
where we assume that the operators $\delta_{n}$ commute pairwise and that $\D(\delta_{1} \dots \delta_{n})$ is independent of the order of the operators $\delta_{n}$ ($1\leq n \leq d$). Then for any multi-index $\alpha = (\alpha_j)_{j=1}^d \in \mathbb{N}_0^d$, the operator $\delta^{\alpha} = \prod_{j=1}^d \delta_j^{\alpha_j}$ and its domain $\D(\delta^{\alpha})$ are well-defined. We equip $\D(\delta^{\alpha})$ with the norm
$$\Vert A \Vert_{\D(\delta^{\alpha})} = \sum_{\beta \leq \alpha} \Vert \delta^{\beta}(A)\Vert_{\A},$$
which is just the (iteratively defined) graph norm on the corresponding preceding derivation algebra $\D(\delta^{\beta})$ ($\vert\beta \vert = \vert \alpha \vert -1$). Furthermore, the Leibniz rule (\ref{Leibniz}) implies the generalized Leibniz rule
$$\delta(AB) = \sum_{\beta \leq \alpha} \binom{\alpha}{\beta} \delta^{\beta}(A)\delta^{\alpha - \beta}(B) \qquad (\forall A,B \in \D(\delta^{\alpha})).$$
This yields  
$$\Vert AB \Vert_{\D(\delta^{\alpha})} \leq C \Vert A \Vert_{\D(\delta^{\alpha})} \Vert B \Vert_{\D(\delta^{\alpha})} \qquad (\forall A,B \in \D(\delta^{\alpha})),$$
where the positive constant $C$ depends only on $\alpha$. Hence (see also \cite[Lemma 3.6]{groeklotz10}) $\D(\delta^{\alpha})$ is an involutive subalgebra of $\A$. By passing on to the equivalent norm 
\begin{equation}\label{eqnorm}
    \vert\Vert A \Vert\vert_{\D(\delta^{\alpha})} := \sup_{\substack{B\in \D(\delta^{\alpha}) \\ \Vert B \Vert_{\D(\delta^{\alpha})} = 1}} \Vert A\cdot B \Vert_{\D(\delta^{\alpha})} \qquad (A \in \D(\delta^{\alpha})),
\end{equation}
we obtain that $\D(\delta^{\alpha})$ is a Banach *-algebra. 

By iterative applications of Proposition \ref{d=1derivationthm}, one can show the following:

\begin{proposition}\label{commutingderivationsthm}\cite[Proposition 3.7]{groeklotz10}) 
Let $\A$ be a symmetric Banach algebra and let $\lbrace \delta_1, \dots , \delta_d \rbrace$ be a set of commuting derivations with $\mathcal{I}_{\A} \in \D(\delta_j)$ for all $1\leq j \leq d$. Then $\D(\delta^{\alpha})$ is inverse-closed in $\A$ for each multi-index $\alpha \in \mathbb{N}_0^d$.
\end{proposition}

Following \cite{groeklotz10}, we may apply the above machinery by considering derivations defined by the commutator with respect to a suitable diagonal matrix. Similar ideas have been also pursued in \cite[Section 5.5]{zbMATH06949773}. Note that the commutator (\ref{commutator}) considered below is essentially the infinitesimal generator of the group defining the functions $f_A$ in Section \ref{The Baskakov-Gohberg-Sjöstrand algebra} and $\check{\widetilde{T}}(h)$ from \cite[Section 5.5]{zbMATH06949773}, respectively.

\begin{theorem}\label{anisotropicA}
Let $\A \subseteq \B(\ell^2(X;\Hil))$ be a Wiener pair and assume that $\A$ is solid. Then for every $\alpha = (\alpha_j)_{j=1}^d \in \mathbb{N}_0^d$, the class of $\B(\Hil)$-valued matrices $A=[A_{k,l}]_{k,l\in X}$ satisfying 
\begin{equation}\label{anidecay1}
    \left\Vert \left[\prod_{j=1}^d (1+\vert k_j - l_j\vert)^{\alpha_j} A_{k,l}\right]_{k,l\in X} \right\Vert_{\A} < \infty 
\end{equation}
is a solid unital Banach *-algebra with respect to the norm (\ref{eqnorm}), which is inverse-closed in $\B(\ell^2(X;\Hil))$ and in particular a symmetric Banach algebra.
\end{theorem}

\begin{proof}
\emph{Step 1.} For $1\leq j \leq d$, let $M_j$ be the $\B(\Hil)$-valued diagonal matrix given by $M_j=\text{diag}[k_j \cdot \mathcal{I}_{\B(\Hil)}]_{k\in X}$, where $k_j$ is the $j$-th coordinate of the index $k = (k_j)_{j=1}^d \in X \subseteq \mathbb{R}^d$. Then the formal commutator 
\begin{equation}\label{commutator}
    \delta_{j}(A) := [M_j,A] = M_j A-AM_j
\end{equation}
has the canonical matrix representation $$\mathbb{M}(\delta_{j}(A)) = [(k_j-l_j)A_{k,l}]_{k,l\in X}.$$

\emph{Step 2.} Fix some $j\in \lbrace 1, \dots, d\rbrace$. Consider the subspace 
$$\D(\delta_j) := \lbrace A\in \A : \delta_j(A) \in \A \rbrace$$
of $\A$. We will show that 
$$\delta_j : \D(\delta_j) \longrightarrow \A$$ 
is a derivation on $\A$. It is quickly verified, that the Leibniz rule (\ref{Leibniz}) is satisfied by $\delta_j$ on $\D(\delta_j)$. It is also clear that $\D(\delta_j)$ is invariant under involution and that $\delta_j(A^*) = \delta_j(A)^*$ for all $A\in \D(\delta_j)$. Hence, in order to show that $\delta_j$ is a derivation on $\A$, we have to verify that $\delta_j$ is a closed operator. To this end, assume that $\lbrace A^{(n)}\rbrace_{n=1}^{\infty} = \lbrace [A_{k,l}^{(n)}]_{k,l\in X}\rbrace_{n=1}^{\infty}$ is a sequence in $\D(\delta_j)$ converging in $\D(\delta_j)$ to some $\B(\Hil)$-valued matrix $A = [A_{k,l}]_{k,l\in X}$ and that $\lbrace \delta_j(A^{(n)})\rbrace_{n=1}^{\infty} = \lbrace [(k_j-l_j)A_{k,l}^{(n)}]_{k,l\in X}\rbrace_{n=1}^{\infty}$ converges in $\A$ to some $\B(\Hil)$-valued matrix $B = [B_{k,l}]_{k,l\in X}$. Then both of these sequences are Cauchy sequences in $\A$ and thus their respective limits $A$ and $B$ are contained in $\A$. Since the norm on $\A$ is solid and hence only depends on operator norms of the entries of the corresponding matrices, this implies that for each $k,l\in X$, the sequence $\lbrace A_{k,l}^{(n)}\rbrace_{n=1}^{\infty}$ is Cauchy sequence in $\B(\Hil)$ converging to $A_{k,l} \in \B(\Hil)$. Consequently, the Cauchy sequence $\lbrace (k_j-l_j)A_{k,l}^{(n)}\rbrace_{n=1}^{\infty}$ converges to $(k_j-l_j)A_{k,l} \in \B(\Hil)$ for each $k,l\in X$. On the other hand, by the same argument we have that $\lbrace (k_j-l_j)A_{k,l}^{(n)}\rbrace_{n=1}^{\infty}$ is a Cauchy sequence in $\B(\Hil)$ converging to $B_{k,l} \in \B(\Hil)$, for each $k,l\in X$. Thus $B_{k,l} = (k_j-l_j)A_{k,l}$ for all $k,l \in X$, which shows that $\delta_j$ is a closed operator. 

\emph{Step 3.} For each $j\in \lbrace 1, \dots, d\rbrace$, $\mathcal{I}_{\A} = \text{diag}[{\mathcal{I}_{\B(\Hil)}}]_{k\in X} \in \D(\delta_j)$, since $\delta_j(\mathcal{I}_{\A}) = 0 \in \A$, and the operators $\delta_j$ commute pairwise. Thus $\lbrace \delta_1, \dots, \delta_d\rbrace$ is a set of commuting derivations. Hence, for each multi-index $\alpha \in \mathbb{N}_0^d$, $\delta^{\alpha}$ and its domain $\D(\delta^{\alpha})$ are well-defined. Furthermore, the associated norm
$$\Vert A \Vert_{\D(\delta^{\alpha})} = \sum_{\beta \leq \alpha} \Vert \delta^{\beta}(A) \Vert_{\A} = \sum_{\beta \leq \alpha} \left\Vert \left[ \prod_{j=1}^d \vert k_j - l_j\vert^{\beta_j} A_{k,l} \right]_{k,l\in X} \right\Vert_{\A}$$
is solid, since the norm on $\A$ is solid. In particular, since $\A$ is symmetric due to Hulanicki's Lemma, Proposition \ref{commutingderivationsthm} yields that $\D(\delta^{\alpha})$ is inverse-closed in $\A$ and consequently also in $\B(\ell^2(X;\Hil))$. Another application of Hulanicki's Lemma yields that $\D(\delta^{\alpha})$ is symmetric.

\emph{Step 4.} Finally, condition 
(\ref{anidecay1}) is equivalent to the condition $[A_{k,l}]_{k,l\in X} \in \D(\delta^{\alpha})$. Indeed, this follows from the solidity of $\A$ and from $\big[\prod_{j=1}^d (k_j - l_j)^{\alpha_j} A_{k,l}\big]_{k,l\in X} = \delta^{\alpha}(A)$. This completes the proof.
\end{proof}

Note that all of our previously proven examples of inverse-closed sub-algebras $\A$ of $\B(\ell^2(X;\Hil))$ are solid. Hence the above theorem applies to each of these cases. We explicitly state the following corollaries of Theorem \ref{anisotropicA} applied to the Wiener pairs $\J_s \subseteq \B(\ell^2(X;\Hil))$ (see Corollary \ref{Jaffardinverseclosed}), $\MS_{\nu}^1 \subseteq \B(\ell^2(X;\Hil))$ (see Theorem \ref{Schurspectraltheorem}) and $\C_{\nu} \subseteq \B(\ell^2(\mathbb{Z}^d;\Hil))$ (see Theorem \ref{Cvinverseclosed}) respectively.

\begin{corollary}\label{anisotropicpolynomial}
For every $s>d$ and $\alpha = (\alpha_j)_{j=1}^d \in \mathbb{N}_0^d$, the class of $\B(\Hil)$-valued matrices $[A_{k,l}]_{k,l\in X}$ satisfying the anisotropic polynomial decay condition
$$\Vert A_{k,l} \Vert \leq C(1+\vert k-l \vert)^{-s} \prod_{j=1}^d (1+\vert k_j - l_j\vert)^{-\alpha_j} \qquad (\forall k,l\in X)$$
is a solid unital Banach *-algebra which is inverse-closed in $\B(\ell^2(X;\Hil))$ and thus symmetric.
\end{corollary}

\begin{corollary}\label{anisotropicschur}
Let $\nu$ be an admissible weight in the sense of Definition \ref{charlyadmissibledef} which satisfies the weak growth condition 
$$\nu(x) \geq C (1+\vert x \vert)^{\delta} \qquad \text{for some } \delta\in (0,1], C>0.$$ Then for every $\alpha = (\alpha_j)_{j=1}^d \in \mathbb{N}_0^d$, the class of $\B(\Hil)$-valued matrices $[A_{k,l}]_{k,l\in X}$ satisfying the anisotropic Schur-type conditions 
\begin{flalign}
&\sup_{k\in X} \sum_{l\in X} \prod_{j=1}^d (1+\vert k_j - l_j \vert)^{\alpha_j} \Vert A_{k,l} \Vert \nu(k-l) <\infty, \notag \\ 
&\sup_{l\in X} \sum_{k\in X} \prod_{j=1}^d (1+\vert k_j - l_j \vert)^{\alpha_j}\Vert A_{k,l} \Vert \nu(k-l) <\infty \notag
\end{flalign}
is a solid unital Banach *-algebra which is inverse-closed in $\B(\ell^2(X;\Hil))$ and thus symmetric.
\end{corollary}

\begin{corollary}\label{anisotropicCv}
For every submultiplicative and symmetric weight satisfying the GRS-condition and every $\alpha = (\alpha_j)_{j=1}^d \in \mathbb{N}_0^d$, the class of $\B(\Hil)$-valued matrices $[A_{k,l}]_{k,l\in \mathbb{Z}^d}$ satisfying the condition
$$\sum_{l\in \mathbb{Z}^d} \prod_{j=1}^d (1+\vert l_j \vert)^{\alpha_j} \sup_{k\in \mathbb{Z}^d} \Vert A_{k,k-l} \Vert \nu(l) <\infty$$
is a solid unital Banach *-algebra which is inverse-closed in $\B(\ell^2(\mathbb{Z}^d;\Hil))$ and thus symmetric.
\end{corollary}


\section*{Acknowledgments}

The authors thank Hans Feichtinger, Karlheinz Gröchenig, Jakob Holböck, Nicki Holighaus, Andreas Klotz, Ilya Krishtal, Arvin Lamando, and Rossen Nenov for their valuable suggestions and related discussions. 
The authors also thank the anonymous reviewer for the many helpful references, comments and suggestions. 

This work is supported by the project P 34624 \emph{"Localized, Fusion and Tensors of Frames"} (LoFT) of the Austrian Science Fund (FWF).

\bibliographystyle{abbrv}

\bibliography{biblioall}
\end{document}